\documentclass[12pt, a4article]{amsart}

\usepackage[top=1.17in, bottom=1.17in, left=1.19in, right=1.19in]{geometry}
\usepackage[colorlinks=true, pdfstartview=FitV, linkcolor=blue, citecolor=blue, urlcolor=blue, breaklinks=true]{hyperref}
\numberwithin{equation}{section}

\usepackage{amssymb,amscd,graphics, amsfonts, multicol}
\usepackage{graphics}
\usepackage{stmaryrd}
\usepackage{DotArrow}
\usepackage{amsmath, amsthm,wasysym}
\usepackage{xypic}
\usepackage{color}

\newcommand{\ui}{\underline{i}}
\newcommand{\uj}{\underline{j}}
\newcommand{\meet}{\mathrel{\text{\raisebox{0.25ex}{\scalebox{0.8}{$\wedge$}}}}}
\newcommand{\join}{\mathrel{\text{\raisebox{0.25ex}{\scalebox{0.8}{$\vee$}}}}}
\newcommand{\mC}{\mathfrak{C}}
\newcommand{\lex}{lex}
\newcommand{\ve}{\varepsilon}
\newcommand{\leqlex}{\leq_{\lex}}
\newcommand{\geqlex}{\geq_{\lex}}
\DeclareMathOperator{\Left}{left}
\DeclareMathOperator{\Right}{right}
\newcommand{\mCright}{\mathfrak{C}_{\Right}}
\newcommand{\mCleft}{\mathfrak{C}_{\Left}}

\theoremstyle{plain}
\newtheorem{theorem}{Theorem}[section]
\newtheorem{coro}[theorem]{Corollary}
\newtheorem{proposition}[theorem]{Proposition}
\newtheorem{lemma}[theorem]{Lemma}

\theoremstyle{definition}

\newtheorem{remark}[theorem]{Remark}
\newtheorem{example}[theorem]{Example}
\newtheorem{definition}[theorem]{Definition}

\title{Schubert valuations on Grassmann varieties}
\author{Rocco Chiriv\`i}
\address{Dipartimento di Matematica e Fisica ``Ennio De Giorgi'', Universit\`a del Salento, Lecce, Italy}
\email{rocco.chirivi@unisalento.it}

\author{Xin Fang} 
\address{Lehrstuhl f\"ur Algebra und Darstellungstheorie, RWTH Aachen University, Pontdriesch 10-16, 52062 Aachen, Germany}
\email{xinfang.math@gmail.com}

\author{Peter Littelmann}
\address{Department Mathematik/Informatik, Universit\"at zu K\"oln, 50931, Cologne, Germany}
\email{peter.littelmann@math.uni-koeln.de}
\begin{document}

\maketitle

\begin{abstract}
The goal of the paper is twofold: on one side it provides an order structure on the set of all maximal chains in the Bruhat poset of Schubert varieties in a Grassmann variety; on the other hand, using this order structure, it works out explicit formulae for the valuation and the Newton-Okounkov body associated to each maximal chain appearing in the framework of Seshadri stratification.

\end{abstract}

\section{Introduction}
In \cite{CFL1}, the authors introduced the notion of a Seshadri stratification on an embedded projective variety.
It can be thought of as a generalization of the Newton-Okounkov theory: the flag of subvarieties
in Newton-Okounkov theory is replaced by a web of subvarieties, indexed by a partially ordered set $A$.
The choice of a sequence of parameters in the Newton-Okounkov theory is replaced by the choice of a 
collection of functions, called \textit{extremal functions}, in the case of a Seshadri stratification. And last, but not least,
the valuation constructed in the framework of Newton-Okounkov theory is replaced
by a quasi-valuation, which is defined as a minimum function over a finite number of valuations,
these are associated to maximal chains in the set $A$.

The structure of the set of all maximal chains in the partially ordered set $A$ and 
the choice of the extremal functions play an important role in the theory of Seshadri stratifications. 
This is the background and the motivation for this article,
in which we concentrate on one example: the Grassmann variety 
$\textrm{Gr}_{d,n}\hookrightarrow \mathbb P(\Lambda^d\mathbb K^n)$, embedded into the projective space 
via the Pl\"ucker embedding. It can be considered as one of the simplest non-trivial examples.

In Section 2, we fix as web of subvarieties for the Seshadri stratification the Schubert varieties. So the set $A$ is
well known in this case, it is the set  $I(d,n)$ of subsets of cardinality $d$ in $\{1,2,\ldots,n\}$ with $d < n$.
Denote by $C(d,n)$ the set of all maximal chains in $I(d,n)$, we want to endow $C(d,n)$  with a poset structure. 
Set $r=d(n-d)$, this is the dimension of the  Grassmann variety, and let $\mathsf{S}_r$ be the symmetric group on $r$ letters. 
We define on the set $C(d,n)$ operators behaving like Kashiwara's crystal operators, and we show that the associated
graph is connected. For two element $x,y\in  \mathsf{S}_r$ denote by $[x,y]_R$ the interval in $\mathsf{S}_r$ with 
respect to the right weak Bruhat order. We use the operators to show (see Theorem~\ref{main_theorem} for a precise formulation)
\par\vskip 5pt\noindent
\textbf{Theorem A.} \textit{
There exists an element $\sigma_{\max}\in \mathsf{S}_r$ and a bijection $C(d,n) \longrightarrow   [e,\sigma_{\max}]_R$
between the set $C(d,n)$ of all maximal chains in $I(d,n)$ and the interval $[e,\sigma_{\max}]_R\subseteq \mathsf{S}_r$.}
\par\vskip 5pt\noindent
The second part of the article concerns the choice of the extremal vectors. A fixed flag $Y^k_\bullet= Y_k\supseteq \ldots\supseteq Y_0=\{p\}$ 
of subvarieties contained in an embedded projective variety $Y\subseteq \mathbb P(V)$,
together with an associated sequence of parameters $\mathbf{u}=(u_r,\ldots,u_1)$ gives rise 
(under certain conditions) to a valuation $\nu^{\mathbf{u}}_{Y_\bullet}$. 
The associated sequence of parameters is far from being unique, in Section~3 
we review how the valuation $\nu^{\mathbf{u}}_{Y_\bullet}$ changes if one replaces
$\mathbf{u}$ by a different associated sequence of parameters $\mathbf{v}$.

In Section 4 we use this to discuss the connection between two different types
of valuations one gets for maximal chains of Schubert varieties on the 
Grassmann variety: one starts with the Seshadri stratification having 
as web of subvarieties the Schubert varieties $\{X(\underline i)\mid \underline i\in I(d,n)\}$,
and as collection of functions the Pl\"ucker coordinates $\{P_{\underline i}\mid \underline i\in I(d,n)\}$.
Each maximal chain $\mathfrak C\in C(d,n)$ gives rise to a valuation $\nu_{\mathfrak C}^s$, a valuation monoid $\Gamma^s_{\nu_\mathfrak C}$ 
and a Newton-Okounkov body $\Delta_{\nu_{\mathfrak{C}}}^s$.  The valuation monoid and the Newton-Okounkov body are usually not easy to describe.  Note that there have already been quite a few   articles on valuations and Newton-Okounkov bodies for Grassmann varieties \cite{Kav15,FN,RW}. Their approaches are related to 
the parametrization of canonical bases or  to cluster theory. We would like to point out that the valuation $\nu_\mathfrak{C}^s$ used in the theory of Seshadri stratifications is different to all the valuations in the abovementioned work. Indeed, as in Example \ref{last:example}, the valuation image of a Pl\"ucker coordinate may have both positive and negative coordinates, which is not the case in the above work.

Another class are the valuations associated to the so-called birational sequences. Again, these valuations
are indexed by maximal chains $\mathfrak C\in C(d,n)$, but one has a different associated sequence of parameters.
We denote the associated valuations by $\nu_{\mathfrak C}^b$ for $\mathfrak C\in C(d,n)$. Let 
 $\Gamma_{\nu_\mathfrak C}^b$ be the corresponding valuation monoid 
and denote by $\Delta_{\nu_{\mathfrak{C}}}^b$ the associated Newton-Okounkov body.   One of the advantages of introducing the notion of birational sequences lies in the fact that, rather than calculating vanishing multiplicities of functions, the valuation monoid can be computed explicitly using basic representation theory of Lie algebras.

The goal of Section 4 is to work out explicit formulae for the valuation image $\nu_\mathfrak{C}^s(P_{\underline{i}})$ of a Pl\"ucker coordinate $P_{\underline{i}}$ under $\nu_\mathfrak{C}^s$. 
The key ingredient which helps us to archive the goal is the birational sequence. Indeed, using basic representation theory of Lie algebras, we obtain a closed formula for the value of a Pl\"ucker coordinate under the valuation $\nu_{\mathfrak{C}_{\mathrm{right}}}^b$ (Proposition \ref{Prop:ValCright}). 
Using this formula, we show that the Newton-Okounkov body $\Delta_{\nu_{\mathfrak{C}_{\mathrm{right}}}}^b$ coincides with the order polytope of a poset (Proposition~\ref{Prop:OrderPolytope}).
Furthermore, when changing the maximal chain from $\mathfrak{C}_{\mathrm{right}}$ to $\mathfrak{D}$, the valuations $\nu_{\mathfrak{C}_{\mathrm{right}}}^b$ and $\nu_{\mathfrak{D}}^b$ are related by the permutation $\sigma_{\mathfrak{D}}$ (Corollary \ref{Coro:Prop:2-move}). What remains is to apply the main result in Section 3 because the two valuations above differ by an upper-triangular unipotent matrix $R$.

We summarize this idea in the following diagram: given a homogeneous function $f$ on $\mathrm{Gr}_{d,n}$,
$$
\xymatrix{
\nu_{\mathfrak{C}_{\mathbf{right}}}^s(f) \ar@{.>}[r] \ar[d]_R & \nu_{\mathfrak{D}}^s(f)\\
\nu_{\mathfrak{C}_{\mathbf{right}}}^b(f) \ar[r]_{\sigma_{\mathfrak{D}}^{-1}} & \nu_{\mathfrak{D}}^b(f) \ar[u]_{R^{-1}}.
}
$$
The dashed arrow, which is our goal, is the composition of other three arrows. As a consequence, we obtain an explicit description of the Newton-Okounkov body 
$\Delta_{\nu_{\mathfrak{D}}}^s$.
\vskip 5pt
\noindent
\textbf{Acknowledgements.} The work of R.C. is partially supported by PRIN 2022S8SSW2 ``Algebraic and geometric aspects of Lie Theory''. The work of X.F. is a contribution to Project-ID 286237555-TRR 195 by the Deutsche Forschungsgemeinschaft (DFG, German Research Foundation). The work of P.L. is partially supported by DFG SFB/Transregio 191 ``Symplektische Strukturen in Geometrie, Algebra und Dynamik''.

\section{A poset structure on maximal chains}\label{section:Schubert:sequences}
This section is of purely combinatorial flavour. We provide
a natural bijection between the set of
maximal chains of Schubert varieties in the Grassmann variety $\textrm{Gr}_{d,n}$, $1\le d<n$, and an
interval $[e,\sigma_{\max}]_R\subseteq 
\mathsf{S}_{r}$ in the symmetric group operating on $r$ letters.
Here $r=d(n-d)$ is the dimension of the Grassman variety and $[\cdot,\cdot]_R$ refers to the interval with 
respect to the right weak Bruhat order on $\mathsf{S}_{r}$.
\subsection{Basic facts}\label{section:basic:facts}

 Let $I(d,n)$ be the set of subsets of cardinality $d$ in $\{1,2,\ldots,n\}$ with $d < n$. We call such a subset a \emph{standard row} of length $d$ and we will always write it as an ordered sequence $\ui = i_1 i_2 \cdots i_d$  with $i_1<\ldots<i_d$. 

We begin by recalling some well known results.
Let $\mathsf{S}_n$ be the symmetric group acting on $n$ letters, and let $\mathsf{S}_d\times 
\mathsf{S}_{n-d}\subseteq \mathsf{S}_n$ 
be the parabolic subgroup of permutations acting only on the first $d$ repectively $(n-d)$ letters.
The set $I(d,n)$ is in bijection with the set 
$\mathsf{S}_n^d$ of \textit{minimal representatives} of elements of the quotient $\mathsf{S}_n/\mathsf{S}_d \times \mathsf{S}_{n-d}$. 
An explicit bijection is given by the map:
\[
\mathsf{S}_n^d \longrightarrow  I(d,n),\quad \sigma \longmapsto \sigma(1) \cdots \sigma(d).
\]
The set $I(d,n)$ is a poset with the partial order defined by:
\[
\ui \leq \uj\quad\textrm{ if and only if }\quad i_1\leq j_1, i_2\leq j_2, \ldots, i_d\leq j_d.
\]
The set $\mathsf{S}_n^d\subseteq \mathsf{S}_n$ 
of minimal representatives is endowed with a partial order too, the restriction of the \textit{Bruhat order} on 
$\mathsf{S}_n$. The identification of $\mathsf{S}_n^d$ with
$I(d,n)$  preserves the partial orders, so we call the partial order
on $I(d,n)$ the \textit{Bruhat order on standard rows}.

A relation $\ui < \uj$ is called a \textit{cover} if $\ui < \uj'\le \uj$ implies 
$\uj'= \uj$. The relation $\ui < \uj$ is a cover for the Bruhat order if and only 
if there exists $1\leq h\leq d$ such that $i_k = j_k$ for all $k\neq h$ and $i_h = j_h + 1$.

The set $I(d,n)$ has a unique minimal and a unique maximal element:
$$
\ui_{\min} = 12\cdots d\quad\textrm{\rm respectively}\quad
\ui_{\max } = (n-d+1) (n-d+2)\cdots (n-1) n.
$$ 
For $\ui\in I(d,n)$, the length $\ell(\ui)$ is defined as
$$
\sum_{h=1}^d (i_h - h) = \sum_{h=1}^d i_h - d(d+1)/2.
$$
The length of $\ui\in I(d,n)$ is the same as the length of the corresponding element in $\mathsf{S}_n^d\subseteq \mathsf{S}_n$,
viewed as an element in the Coxeter group $\mathsf{S}_n$. In particular $\ell(\ui_{\min}) = 0$
 and we set $r := \ell(\ui_{\max}) = d(n-d)$. 

The poset $I(d,n)$ can be made a distributive lattice by setting:
\[
\ui \meet \uj = (\min\{i_1,j_1\})(\min\{i_2, j_2\})\cdots(\min\{i_d,j_d\})
\]
and
\[
\ui \join \uj = (\max\{i_1,j_1\})(\max\{i_2, j_2\})\cdots(\max\{i_d,j_d\}).
\]
Given two strings, i.e. ordered sequences $\underline{s}$, $\underline{t}$ in some alphabet, we denote by 
$\underline{s}*\underline{t}$ their concatenation. The following is useful for inductive arguments.
\begin{proposition}\label{proposition_posetDecomposition}
We have
\[
I(d, n) \, = \, I(d, n-1) \, \sqcup\, (I(d-1, n-1) * n).
\]
In particular $I(d, n-1)$ is the subset of $I(d,n)$ consisting of rows not containing $n$, and the set $I(d-1,n-1) * n$ is the subset of 
$I(d,n)$ consisting of those standard rows containing $n$.
\end{proposition}
\begin{proof} The second part is clear and the first part follows by the second one.
\end{proof}
 We sometimes need a total order refining the partial order ``$\le$'', we take the lexicographic order, i.e. $\ui \le_{lex} \uj$ if there exists $1\le t\le d$ such that $i_s=j_s$ for $s<t$ and $i_t<j_t$.
\subsection{Maximal chains}
A \textit{maximal chain} in $I(d,n)$ is a totally ordered subset which is maximal with this property with respect to the inclusion of subsets.
In more geometric terms: to fix a maximal chain in
$I(d,n)$ means one fixes a chain of Schubert varieties, subsequently contained in each other of codimension one, and maximal with this property:
$$
\textrm{\rm Gr}_{d,n}=X({\ui_{\max}})\supsetneq X({\underline i}_{r-1}) \supsetneq\ldots \supsetneq 
X({\underline i}_1) \supsetneq 
X({\underline i}_{\min}).
$$
\begin{remark}\label{remark_gradedPoset} The poset $I(d,n)$ is graded and all maximal chains
\[
\mC:\ui_r = \ui_{\max} > \ui_{r-1} > \cdots > \ui_1 > \ui_0 = \ui_{\min}
\]
in $I(d,n)$ have the same length $r = d (n - d)$.
\end{remark}

\begin{definition}\label{definition_MaxChains} We denote by $C(d,n)$ the set of all maximal chains in $I(d,n)$. We define a total order on  $C(d,n)$ as follows: for two maximal chains $\mC:\ui_r > \ui_{r-1} > \cdots > \ui_0$ and $\mathfrak{D}:\uj_r > \uj_{r-1} > \cdots > \uj_0$ we write $ \mathfrak{D}\geqlex \mC$ if
\[
\uj_r * \uj_{r-1} * \cdots * \uj_0 \geqlex \ui_r * \ui_{r-1} * \cdots * \ui_0 ,
\]
where $\leqlex$ is the lexicographic order on strings in the alphabet $\{1,2,\ldots,n\}$. In such a case, we say that $\mC$ is \emph{on the right of} $\mathfrak{D}$ and that $\mathfrak{D}$ is \emph{on the left} of $\mC$.
\end{definition}
The wordings ``on the left'' and ``on the right'' are due to the fact that we usually depict the poset $I(d,n)$ as diagram writing greater elements on the left. 
\begin{example}\label{example:I24}
The poset $I(2,4)$ has two maximal chains:
$\mathfrak{D}: 34 > 24 > 23 > 13 > 12$ and 
$\mC: 34 > 24 > 14 > 13 > 12$, and $\mathfrak{D}$ 
is greater than $\mC$, so $\mC$ is on the right of $\mathfrak{D}$.
\[
\xymatrix{
34\ar@{-}[dr]\\
& 24\ar@{-}[dl]\ar@{-}[dr]\\
23\ar@{-}[dr] & & 14\ar@{-}[dl]\\
& 13\ar@{-}[dl]\\
12}
\]
\end{example}
\begin{definition}\label{remark_minimalMaximalChainsdefn}
We denote the $\leqlex$--lowest chain in $C(d,n)$ by $\mCright$, it is the rightmost maximal chain. The $\leqlex$--greatest maximal chain
is denoted by $\mCleft$, it is the leftmost maximal chain.
\end{definition}
\begin{remark}\label{remark_minimalMaximalChains}
 The chain $\mCright$ is clearly obtained by starting with the maximal element $\ui_{\max} = (n-d+1)(n-d+2)\cdots n$ and: lowering its first entry till $1$, then lowering its second entry till $2$ and so on. We get for $\mCright$ the chain  
\[
\begin{array}{rcl}
\mCright &:& {\scriptstyle \ui_{\max}= (n-d+1)(n-d+2)\cdots n} > {\scriptstyle (n-d)(n-d+2)\cdots n} > {\scriptstyle (n-d-1)(n-d+2)\cdots n} > \cdots > {\scriptstyle 1(n-d+2)\cdots n}\\
&&> {\scriptstyle 1(n-d+1)(n-d+3)\cdots n} > {\scriptstyle 1(n-d)(n-d+3)\cdots n} > \cdots > {\scriptstyle 12(n+d+3)\cdots n} >\\
&&\hfil \vdots\\
&&> {\scriptstyle 12\cdots (d-1)n} > {\scriptstyle 12\cdots (d-1)(n-1)} >\cdots > {\scriptstyle 12\cdots (d-1)d = \ui_{\min}}.
\end{array}
\]
The chain $\mCleft$ is obtained by starting with $\ui_{\max}$ and: lowering the first entry once, then lowering the second entry once, and continuing in this way till lowering the last entry, then lowering again the first entry and so on. We get:
\[
\begin{array}{rcl}
\mCleft&:& {\scriptstyle \ui_{\max} = (n-d+1)(n-d+2)\cdots n} > {\scriptstyle (n-d)(n-d+2)\cdots n} > {\scriptstyle (n-d)(n-d+1)(n-d+3)\cdots n} > \cdots \\
&&\hfill \cdots > 
{\scriptstyle (n-d)(n-d+1)(n-d+2)\cdots(n-1)}\\
&& > {\scriptstyle (n-d-1)(n-d+1)(n-d+2)\cdots(n-1)} > {\scriptstyle (n-d-1)(n-d)(n-d+2)\cdots(n-1)} > \cdots \\
&&\hfill \cdots > 
{\scriptstyle (n-d-1)(n-d)(n-d+1)\cdots(n-2)}\\
&& \hfil\vdots\\
&& > {\scriptstyle 234\cdots d(d+1)} > {\scriptstyle 134\cdots d(d+1)} > {\scriptstyle 124\cdots d(d+1)} > \cdots > {\scriptstyle 12\cdots (d-1)d = \ui_{\min}}.
\end{array}
\]
\end{remark}
The maximal element in $I(d,n-1)$ is  $\ui_{\max, n-1}:=(n-d)(n-d+1)\cdots (n-1)$. Let $M\subseteq I(d,n)$
be the subset: $M=\{\ui\in I(d,n)\mid \ell(\ui)=
\ell(\ui_{\max, n-1}) \}$ of all elements of the same length as  $\ui_{\max, n-1}$.
\begin{lemma}\label{lemma_maximalChainContainingSubMax}
\begin{itemize}
\item[(i)]
$\ui_{\max, n-1}$ is maximal in $M$ with respect to the lexicographic order $\ge_{lex}$, i.e.:  $\ui_{\max, n-1}\ge_{lex} \uj$ for all $\uj\in M$.
\item[(ii)]  Consider the interval $[\ui_{\max, n-1},\ui_{\max}]_{lex}$ with respect to the lexicographic order. This set is equal to the interval $[\ui_{\max, n-1},\ui_{\max}]$ with respect to the Bruhat order, it is
also totally ordered with respect to the Bruhat order, and its elements are precisely the elements in the first two rows for $\mCleft$ above.  
\end{itemize}
\end{lemma}
\begin{proof} 
Let $\uj = j_1j_2\cdots j_d$ and $\ui_{\max, n-1} = i_1i_2\cdots i_d$.
If $\uj\geqlex\ui_{\max, n-1}$, then $j_1\geq n-d$, hence $j_2\geq n-d+1$, $j_3\geq n-d+2$, ... So 
we have $j_k \geq i_k$ for all $k$. But if $\uj$ has the same length as $\ui_{\max, n-1}$, then 
$\sum j_k = (n-d) + (n-d+1) + \cdots + (n-1)$. We get $j_k = n-d+k - 1$ for all $k$, hence $\uj = \ui_{\max, n-1}$. This proves (i).

The claim (ii) is evident by the definition of the lexicographic order on rows.
\end{proof}

\begin{proposition}\label{proposition_chainDecomposition}
Let $\mC:\ui_r > \ui_{r-1} > \cdots > \ui_0$ be a maximal chain in $I(d,n)$. Then there exists $0 \leq k < r$ 
such that $\ui_h \in I(d, n-1)$ for all $h \leq k$ and $\ui_h \in I(d-1, n-1)*n$ for all $h > k$. Moreover, if 
$\ui_k = i_{1,k}i_{2,k}\cdots i_{d-1,k}i_{d,k}$, then $i_{d,k} = n-1$ and $\ui_{k+1} = i_{1,k}i_{2,k}\cdots i_{d-1,k} * n$.
\end{proposition}
\begin{proof} Since $\ui_0 = \ui_{\min}\in I(d,n-1)$ and $i_{max}\in I(d-1,n-1)*n$, there exists a maximal $k$ such that 
$$\ui_k = i_{1,k}i_{2,k}\cdots i_{d-1,k}i_{d,k}\in I(d,n-1).$$ 
We have $i_{d,k} < n$ and clearly $k<r$, hence the last entry of each $\ui_h$, $h = 0,1,\ldots,k$, is less than $n$ since $\ui_h \leq \ui_k$ for all $h = 0,1,\ldots,k$. This shows that $\ui_h \in I(d, n-1)$ for all $h = 0,1,\ldots,k$. 

Now $\ui_{k+1} \in I(d-1, n-1)*n$, in particular its last entry is $n$. Since $\ui_{k+1}$ covers $\ui_k$, these two rows must differ by exactly one entry: the last one. This proves the last statement of the proposition.

Finally, since $\ui_{k+1}\leq\ui_h$ for all $h = k+1, \ldots, r$, it is clear that the last entry of each of these rows is $n$; hence $\ui_h\in I(d-1,n-1)*n$ for all $h = k+1, \ldots, r$.
\end{proof}

\subsection{Root operators on standard rows}
We  define a class of operators on 
the set $I(d,n)$ of standard rows. In analogy 
to the Kashiwara operators on the path model and on crystal graphs,
we call them \textit{root operators}. 
\begin{definition}\label{definition_rootOperator}
Let $1\leq s\leq n-1$, $1\leq h\leq d$ with $n-d+h>s\geq h$. For each such pair $(s,h)$ we define a 
map $f_{s,h} : I(d,n)\sqcup \{0\} \longmapsto I(d,n)\sqcup \{0\}$, called a \emph{lowering root operator}. 
Here we consider $0$ as a special symbol.
Set $f_{s,h}(0) = 0$ and for $\ui = i_1i_2\cdots i_d\in I(d,n)$
\[
f_{s,h} (\ui) = \left\{
\begin{array}{ll}
i_1i_2\cdots i_{h-1}\, (s + 1)\, i_{h+1}\ldots i_d, & \textrm{if }i_h = s\textrm{ and either }h = d \textrm{ or } i_{h+1} \geq s+2;\\
0, & \textrm{otherwise.}
\end{array}
\right.
\]
In the same way we define the \emph{rising root operator} $e_{s,h} : I(d,n)\sqcup \{0\} \longmapsto I(d,n)\sqcup \{0\}$
\[
e_{s,h} (\ui) = \left\{
\begin{array}{ll}
i_1i_2\cdots i_{h-1}\, s\, i_{h+1}\ldots i_d, & \textrm{if }i_h = s + 1\textrm{ and either }h = 1 \textrm{ or } i_{h-1} \leq s - 1;\\
0, & \textrm{otherwise.}
\end{array}
\right.
\]
\end{definition}
\begin{remark}\label{remark_loweringRisingRow}
It is clear by definition that if $f_{s,h}(\ui) \neq 0$ then $e_{s,h}f_{s,h}(\ui) = \ui$ and, analogously, 
if $e_{s,h}(\ui) \neq 0$ then $f_{s,h}e_{s,h}(\ui) = \ui$.
\end{remark}
\begin{remark}\label{number_remark_loweringRisingRow}
For $h=1$ we have $(n-d+1)-1$ operators $f_{s,1}$, for $h=2$ we  have $(n-d+2)-2$ operators $f_{s,2}$, and so on. It follows:
we have exactly $r=d(n-d)$ operators of the form $f_{s,h}$ and, similarly,  $r$ operators of the form $e_{s,h}$.
\end{remark}
\begin{remark}
The operators $f_{s,h}$ and $e_{s,h}$ are 
a ``refinement" of the usual crystal
graph operators on Young tableaux in the following sense: if we denote the lowering
crystal graph operators by $f_s$, $1\le s\le n-1$, then either $f_s(\underline{i}) = 
f_{s,h}(\underline{i})$ if $f_s$ replaces
$i_h= s$ by $s + 1$ in $\underline{i}$ (and this $h$ is the unique for which $f_{s,h}(\underline{i}) \not= 0$), or 
$f_s(\underline{i}) = 0$,
and $f_{s;h}(\underline{i}) = 0$ for all $h$. 
A similar formula holds also for the lowing crystal graph operators.
\end{remark}
We study the effect of products of the operators of degree 2:
\begin{proposition}\label{proposition_rootOperatorFirst} 
Let $\ui\in I(d,n)$.
\begin{itemize}
\item[(i)] If $\vert h-k\vert >1$, or $\vert h-k\vert = 1$ and $\vert s-t\vert>1$, then $f_{s,h}f_{t,k}(\ui) = f_{t,k}f_{s,h}(\ui)$.
\item[(ii)] If $|s - t| > 1$ or $s=t$, then $f_{s,h}f_{t,h}(\ui) = 0 = f_{t,h}f_{s,h}(\ui)$.
\item[(iii)] $f_{s,h}f_{s+1,h}(\ui) = 0$ and, for a row $\ui = i_1i_2\cdots i_d$, we have $f_{s+1,h}f_{s,h}(\ui) \neq 0$ 
if and only if $i_h = s$ and either $s=d$ or $i_{h+1}>s+2$.
\item[(iv)] If $\vert h-k\vert = 1$ and $\vert s-t\vert=1$, then 
$f_{s+1,h+1}f_{s,h} (\ui), f_{s,h+1}f_{s+1,h} (\ui),f_{s+1,h}f_{s,h+1}(\ui)=0$, and, for a row 
$\ui = i_1i_2\cdots i_d$, we have $f_{s,h}f_{s+1,h+1}(\ui) \neq 0$ if 
and only if $i_h = s$, $i_{h+1} = s+1$ and either $h+1=d$ or $i_{h+2}>s+2$.
\end{itemize}
\end{proposition}
\begin{proof} We start by proving (i). Since $h\neq k$, the operators $f_{s,h}$ and $f_{t,k}$ change 
different entries of the row $\ui$ if they are not $0$. If $\vert h-k\vert>1$, then one has clearly
$f_{s,h}f_{t,k}(\ui) = f_{t,k}f_{s,h}(\ui)$. If $\vert h-k\vert=1$, then we know in addition: $\vert s-t\vert>1$. 
It follows applying one combination  of the operators to $\ui$ gives zero if and only if it is so for the other. 
And, if applying a combination to $\ui$ gives something nonzero, then $f_{s,h}(f_{t,k}(\ui)) = f_{t,k}(f_{s,h}(\ui))$.

If $s = t$ in (ii), then $f_{s,h}^2 = 0$ by  definition. If $|s - t| > 1$ then either $f_{t,h}(\ui) = 0$ or 
$f_{s,h}(f_{t,h}(\ui)) = 0$, so both, $f_{s,h}f_{t,h}(\ui)$ and $f_{t,h}f_{s,h}(\ui)$, are $0$.

The properties (iii) and (iv) are evident by definition.
\end{proof}

\begin{proposition}\label{proposition_rootOperatorSecond}
\begin{itemize}
\item[(i)] $\ui \leq \uj$ if and only if $\uj$ can be obtained from $\ui$ by applying a 
suitable monomial in the root operators  $f_{s,h}$ to $\ui$.
\item[(ii)] Let $\mC:\ui_{\max} = \ui_r > \ui_{r-1} > \cdots > \ui_1 > \ui_0 = \ui_{\min}$ be a 
maximal chain in $I(d,n)$. Then, for each $t = 1, \ldots, r$, there exists $s_t,h_t$ such that 
$\ui_t = f_{s_t,h_t}(\ui_{t-1})$. Moreover, each root operator $f_{s,h}$,
$1\le s\le n-1$, $1\le h\le d$ with $n-d+h>s\ge h$, appears in the list $(f_{s_t,h_t}\,|\,t = 1, \ldots, r)$, 
and it appears exactly once.
\item[(iii)] In particular, let $\mC:\ui_{\max} = \ui_r > \ui_{r-1} > \cdots > \ui_1 > \ui_0 = \ui_{\min}$ 
be a maximal chain and suppose that $\ui_h\in I(d,n-1)$ for $h = 1, 2, \ldots, t$ and 
$\ui_h\in I(d-1,n-1)*n$ for all $t < h \leq r$; then $\ui_{t+1} = f_{n-1,d}(\ui_t)$; 
in other words, $s_{t+1} = n-1$ and $h_{t+1} = d$.
\end{itemize}
\end{proposition}
\begin{proof} The property (i) follows immediately  
by the basic facts recalled in Section~\ref{section:basic:facts}. Reformulated in terms of operators one has: $\ui<\uj$ is a cover
if and only if $\ui=f_{s,h}(\uj)$ for some $1\le s\le n-1$, $1\le h\le d$, $n-d+h>s\ge h$.

The existence of the operators $f_{s_t,h_t}$ in (ii) follows by (i). Each root operator $f_{s,h}$ with $1\le s\le n-1$, $1\le h\le d$, $n-d+h>s\ge h$
must appear in the list $(f_{s_t,h_t}\,|\,t = 1, \ldots, r)$ once and only once since the $h$--th entry of $\ui_{\min}$, which is equal to $h$, 
must reach the $h$--th entry of $\ui_{\max}$, 
which is $n-d + h$, in increasing steps. This proves (ii).

Finally, (iii) is just a reformulation of the last statement of Proposition \ref{proposition_chainDecomposition}.
\end{proof}

\subsection{Maximal chains and the symmetric group on \textit{r} letters}
By part $(ii)$ of Proposition~\ref{proposition_rootOperatorSecond},
we can associate to a maximal chain 
$\mC:\ui_{\max} = \ui_r > \ui_{r-1} > \cdots > \ui_0 = \ui_{\min}$ in $I(d,n)$ an ordered list 
$(f_{s_t,h_t}\,|\,t = 1, \ldots, r)$, 
i.e. $\ui_t=f_{s_t,h_t}(\ui_{t-1})$. 
Each of the different $r$ operators $f_{s,h}$ with $1\le s\le n-1$, $1\le h\le d$, $n-d+h>s\ge h$, appears precisely once. Hence replacing 
one maximal chain by another corresponds in terms of these
ordered lists of operators to applying a permutation to the entries of this
ordered list. To make this more precise, we fix a kind of \textit{base point}, the chain $\mCright$. 

\begin{definition}\label{definition_rootOperatorIndexing}
Let $(f_{s_t,h_t}\,|\,t = 1, \ldots, r)$ be the ordered
list of operators (Proposition~\ref{proposition_rootOperatorSecond}) 
associated to the maximal chain $\mCright$.
For $t = 1, \ldots, r$, we write  
$\widetilde{f}_t $ for $f_{s_t,h_t}$ on $I(d,n)\sqcup\{0\}$.
We call $\widetilde{f}_1,\ldots,\widetilde{f}_{r}$ the \emph{indexed root operators}.
\end{definition}

Given another maximal chain 
$\mC: \ui_{r} > \cdots >\ui_{0}$, 
by part (ii) of Proposition \ref{proposition_rootOperatorSecond} we have a list 
$(f_{s_t,h_t}\,|\, t = 1,\ldots,r)$ of root operators such 
that $\ui_t = f_{s_t,h_t}(\ui_{t-1})$. This list is  a permutation of the list associated to $\mC_{\mathrm{right}}$.
So there exists a unique permutation $\sigma = 
\sigma_\mC \in \mathsf{S}_{r}$ (the symmetric group
acting on the set $\{1,2,\ldots, r\}$) 
such that: $f_{s_t,h_t} = \widetilde{f}_{\sigma(t)}$.

\begin{definition}\label{definition_chainPermutation}
We call $\sigma_\mC$ the \emph{permutation} associated to $\mC$.
\end{definition}
\begin{remark}\label{remark_bjornerWachs}
The permutation $\sigma_\mathfrak C$ is the labeling of $\mathfrak C$ as defined by
Bj\"orner and Wachs in \cite{BW} if we fix as reduced expression
$$
s_{n-d}s_{n-d-1}\cdots s_1 s_{n-d+1}s_{n-d}\cdots s_2\cdots s_{n-2}s_{n-3}\cdots s_{d-1} s_{n-1}s_{n-2} \cdots s_d
$$
for the longest element $\underline{i}_{\max}\in\mathsf S^d_n\subseteq \mathsf S_n$,
where $s_i := (i, i + 1)$  for $i = 1,\ldots,n-1.$
\end{remark}

\begin{lemma}\label{lemma_indexedOperators}
\begin{itemize}
\item[(i)] The permutation associated to $\mCright$ is the identity of $\mathsf S_{r}$.
\item[(ii)] The map  $C(d,n)\rightarrow \mathsf{S}_{r}$, 
$\mC\longmapsto\sigma_\mC$ is injective.
\item[(iii)] If $\widetilde{f}_t = f_{s_t,h_t}$, $t = 1, \ldots,r$, as in Definition \ref{definition_rootOperatorIndexing}, and $t_1\leq t_2$ then $h_{t_1} \geq h_{t_2}$.
\end{itemize}
\end{lemma}
\begin{proof}
The property (i) is clear by the definition. The map $\mC\longmapsto \sigma_\mC$ is clearly injective because:  
from the permutation one can reconstruct the ordered 
list of root operators, which gives back the maximal chain $\mC$.

Finally, in order to prove (iii), just look at $\mCright:\ui_r > \cdots > \ui_0$ and note that, as we move from $\ui_0$ to $\ui_r$, we first increase the last $d$--th entry as much as possible, then we increase the $(d-1)$--entry and so.
\end{proof}

\subsection{Root operators on maximal chains}
By Lemma \ref{lemma_indexedOperators}, one can identify $C(d,n)$ with a subset of $S_{r}$, but we 
do not have an action of the symmetric group on the set of maximal chains $C(d,n)$. To overcome 
these limitations, in this section we define certain operators on the set $C(d,n)$. To define the 
operators, we need the following lemma.
\begin{lemma}\label{lemma_intervalLengthTwo} For an interval $\left[\,\ui,\,\uj\,\right]$ 
(with respect to the Bruhat order) of length $2$ (i.e. $\ell(\uj) = \ell(\ui)+2$) in $I(d,n)$, we have two possibilities:
\begin{itemize}
\item[(A)] either $\left[\,\ui,\, \uj\,\right] = \left\{\ui,\ui',\uj\right\}$ with $\ui < \ui' < \uj$,
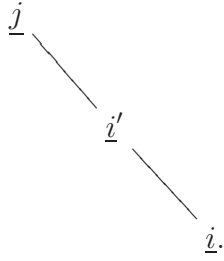
\[
\xymatrix{
\uj\ar@{-}[dr]\\
& \ui'\ar@{-}[dr]\\
& & \ui.        
}
\]
\item[(B)] or $\left[\,\ui,\, \uj\,\right] = \left\{\ui,\ui', \uj',\uj\right\}$ with $\ui < \ui', \uj' < \uj$ and $\ui'\meet \uj' = \ui$, $\ui'\join \uj' = \uj$,
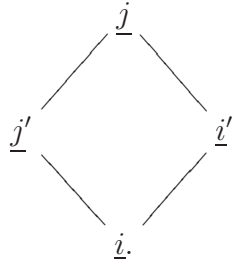
\[
\xymatrix{
& \uj\ar@{-}[dr]\ar@{-}[dl]\\
\uj'\ar@{-}[dr] & & \ui'\ar@{-}[dl]\\
& \ui.        
}
\]
\end{itemize}
\end{lemma}
\begin{proof} Let $\ui = i_1i_2\cdots i_d$ and $\uj = j_1j_2\cdots j_d$. By (i) in 
Proposition \ref{proposition_rootOperatorSecond}, there exists $f_{s,h}$ and 
$f_{t,k}$ such that $\uj = f_{s,h}\,f_{t,k}\,(\ui)$. If $h = k$ then, by (ii) and (iii) in 
Proposition \ref{proposition_rootOperatorFirst}, $s = t + 1$ and
\[
\left[\,\ui,\, \uj\,\right] = \left\{\,\ui < \ui' = f_{t,h}\,(\ui) < \uj = f_{t+1, h}\, f_{t,h}\,(\ui\}))\,\right\}.
\]
If $h\neq k$, then we are either in the case (i) of Proposition \ref{proposition_rootOperatorFirst} and hence
\[
\left[\,\ui,\, \uj\,\right] = \left\{\,\ui < \ui' = f_{t,h}(\ui), \, \uj' = f_{s,k}(\ui) < \uj = f_{s,k}\, f_{t, h}(\ui))\,\right\}
\]
and the last statement in (B) is clear. Or we are  in the case (iv) of Proposition \ref{proposition_rootOperatorFirst} and hence
\[
\left[\,\ui,\, \uj\,\right] = \left\{\,\ui < \ui' = f_{s+1,h+1}(\ui) < \uj = f_{s,h}\, f_{s+1,h+1}\,(\ui)\,\right\}.
\]
\end{proof}

\begin{definition}\label{definition_flatPeak}
Let $\mC:\ui_{\max} = \ui_r > \ui_{r-1} > \cdots > \ui_1 > \ui_0 = \ui_{\min}$ be a maximal chain in $I(d,n)$. We say that $\ui_t$ is:
\begin{itemize}
\item[(F)] \emph{flat} in $\mC$ if $t=0$, or $t = r$, or $1\leq t\leq r-1$ and the interval $[\ui_{t-1},\ui_{t+1}]$ has cardinality $3$ (hence it is as in (A) of Lemma~\ref{lemma_intervalLengthTwo}),
\item[(P)] a \emph{peak} in $\mC$ if $1\leq t< r$, and the interval $[\ui_{t-1},\ui_{t+1}]= \{\ui_{t-1},\ui_t,\ui'_t, \ui_{t+1}\}$ 
has cardinality $4$ (hence it is as in (B) of Lemma~\ref{lemma_intervalLengthTwo}).

The peak is called a \emph{right peak} if $\ui'_t >_{\lex} \ui_t$, it is called a \emph{left peak} if $\ui_t >_{\lex} \ui'_t$.
In both cases we call $\ui'_t$ the \emph{opposite} peak of $\ui_t$ in $\mC$.
\end{itemize}
\end{definition}
\begin{definition}\label{definition_rootOperatorMaxChain}
For $1\leq t\leq r-1$, we define a map $F_t:C(d,n)\sqcup\{0\}\longrightarrow C(d,n)\sqcup\{0\}$, where we consider $0$ as a special symbol. We set $F_t(0) = 0$ for all $t=1,\ldots,r-1$.

Now let $1\leq t\leq r-1$ and let $\mC:\ui_r > \ui_{r-1} > \cdots > \ui_1 > \ui_0$ be a maximal chain in $I(d,n)$.  If $\ui_t$ is a right peak in $\mC$, the we set
\[
F_t(\mC) = \ui_r > \cdots > \ui_{t+1} > \ui'_t > \underline{i
}_{t-1} > \cdots > \ui_0,
\]
where $\ui'_t$ is its opposite peak. Otherwise, i.e. if $\ui_t$ is flat or a left peak, we set $F_t(\mC) = 0$. We call $F_t$ a \emph{lowering root operator}.

In the same way we define the \emph{rising root operator} $E_t:C(d,n)\sqcup\{0\}\longrightarrow C(d,n)\sqcup\{0\}$. Set $E_t(0) = 0$ for all $t=1,\ldots,r-1$. For $1\leq t \leq r-1$ 
and a maximal chain $\mC:\ui_r > \ui_{r-1} > \cdots > \ui_1 > \ui_0$ we set
\[
E_t(\mC) = \ui_r > \cdots > \ui_{t+1} > \ui'_t > \underline{i}_{t-1} > \cdots > \ui_0
\]
if $\ui_t$ is a left peak in $\mC$, where $\ui'_t$ is its opposite peak. Otherwise we set $E_t(\mC) = 0$.
\end{definition}

\begin{remark}
The standard rows $\ui_0$ and $\ui_r$ are never a peak.
\end{remark}

\begin{remark}\label{remark_loweringRisingChain}
Suppose $1\le t\le r-1$.
It is clear by definition and Lemma \ref{lemma_intervalLengthTwo} that if $F_t(\mC) \neq 0$, 
then $E_tF_t(\mC) = \mC$, and if $E_t(\mC) \neq 0$, then $F_tE_t(\mC) = \mC$.
\end{remark}
\begin{example}
Recall that (Example~\ref{example:I24}) there are two maximal chains in $I(2,4)$.
As simple example, we see the unique non zero lowering operator in case:
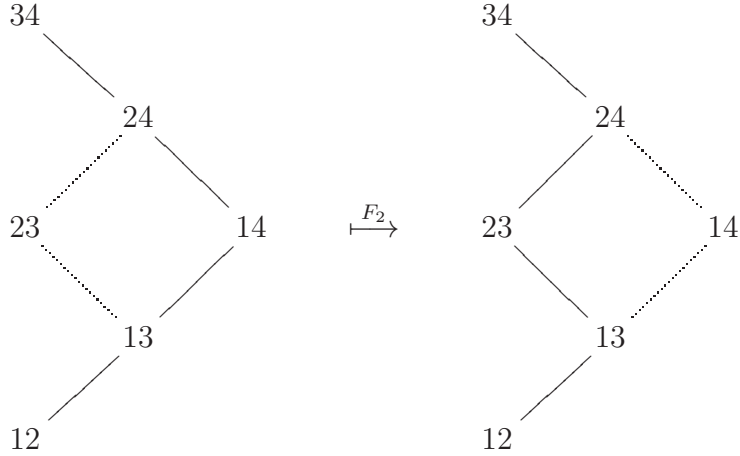
\[
\xymatrix{
34\ar@{-}[dr] & & & & 34\ar@{-}[dr]\\
& 24\ar@{.}[dl]\ar@{-}[dr] & & & & 24\ar@{-}[dl]\ar@{.}[dr]\\
23\ar@{.}[dr] & & 14\ar@{-}[dl] & \stackrel{F_2}{\longmapsto} & 23\ar@{-}[dr] & & 14\ar@{.}[dl]\\
& 13\ar@{-}[dl] & & & & 13\ar@{-}[dl]\\
12 & & & & 12\\
}
\]
\end{example}
We denote by $s_1, \ldots, s_{r-1}$ the simple reflection of the symmetric group $\mathsf{S}_{r}$ 
and by $\leq_R$ the right weak Bruhat order on $\mathsf{S}_{r}$. I.e. $u\le_R v$
if and only if some initial substring of some reduced word for $v$ is a (reduced) word for $u$.
The following simple lemma will be needed later.
\begin{lemma}\label{lemma_weakRightOrder}
Let $1\le t\le r-1$ and
let $\sigma\in \mathsf{S}_{r}$. Then $\sigma s_t >_R \sigma$ (resp. $\sigma s_t <_R \sigma$), if and only if $\sigma(t) < \sigma(t+1)$ (resp. $\sigma(t) > \sigma(t+1)$).
\end{lemma}
\begin{proof} The permutation $\sigma' := \sigma s_t$ is comparable with $\sigma$ with respect to the standard Bruhat order. Since $\sigma$ is obtained by multiplying on the right $\sigma$ by a simple reflection, we have $\sigma' >_R \sigma$ if and only if $\sigma' > \sigma$ with respect to the Bruhat order, otherwise we have $\sigma' <_R \sigma$.
Now, consider the one-line presentation $w = (\sigma(1), \sigma(2), \cdots, \sigma(r-1))$ of $\sigma$ and let $w'$ be that of $\sigma'$. Note that $w'$ and $w$ coincide except that the entries $\sigma(t)$ and $\sigma(t+1)$ are swapped. But then $\sigma'>\sigma$ if and only if $\sigma(t) < \sigma(t+1)$ (see Example 2, Section 5.9 in \cite{Hum}).
\end{proof}

\begin{proposition}\label{proposition_operatorPermutation1}
Let $\mC$ be a maximal chain in $I(d,n)$.
\begin{itemize}
\item[(i)] If $F_t(\mC)\neq 0$, then $\sigma_{F_t(\mC)} = \sigma_\mC s_t$ and $\ell(\sigma_{F_t(\mC)})= \ell(\sigma_\mC) + 1$. In particular, one has 
$\sigma_{F_t(\mC)} >_R \sigma_\mC$.
\item[(ii)] If $E_t(\mC)\neq 0$, then $\sigma_{E_t(\mC)} = \sigma_\mC s_t$ and $\ell(\sigma_{E_t(\mC)})= \ell(\sigma_\mC) - 1$. In particular, one has 
$\sigma_{E_t(\mC)} <_R \sigma_\mC$.
\end{itemize}
\end{proposition}
\begin{proof}
We prove (i), the proof of (ii) is similar. Let $\mC:\ui_r > \ui_{r-1} > \cdots > \ui_1 > \ui_0$ and $\sigma := \sigma(\mC)$. 
Since $F_t(\mC)\neq 0$, the row $\ui_t$ is a right peak in $\mC$; let $\ui'_t$ be its opposite peak. Note that necessarily 
$r>t > 0$. We have the following situation along $\mC$:
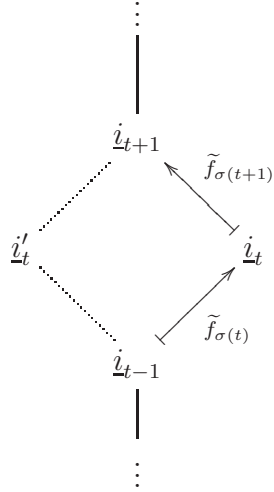
\[
\xymatrix{
 & \vdots\ar@{-}[d]\\
 & \ui_{t+1}\ar@{.}[dl]\\
 \ui'_t & & \ui_t\ar@{|->}[ul]_{\widetilde{f}_{\sigma(t+1)}}\\
 & \ui_{t-1}\ar@{.}[ul]\ar@{|->}[ur]_{\widetilde{f}_{\sigma(t)}}\\
 & \vdots\ar@{-}[u].
}
\]
Since $\ui_t\meet\ui'_t = \ui_{t-1}$ and $\ui_t\join\ui'_t = \ui_{t+1}$ by Lemma \ref{lemma_intervalLengthTwo}, it is clear that we have the following situation along $F_t(\mC)$:
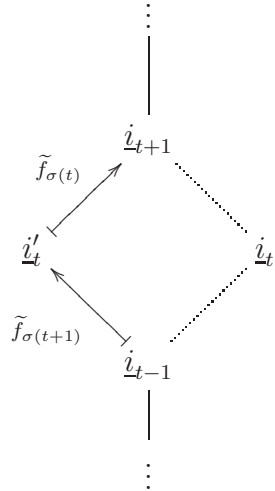
\[
\xymatrix{
 & \vdots\ar@{-}[d]\\
 & \ui_{t+1}\ar@{.}[dr]\\
 \ui'_t\ar@{|->}[ur]^{\widetilde{f}_{\sigma(t)}} & & \ui_t\\
 & \ui_{t-1}\ar@{.}[ur]\ar@{|->}[ul]^{\widetilde{f}_{\sigma(t+1)}}\\
 & \vdots\ar@{-}[u].
}
\]
Set $\sigma' := \sigma_{F_t(\mC)}$.
The two ordered lists $(\widetilde{f}_{\sigma'(s)})_{s=1,\ldots,r}$ and $(\widetilde{f}_{\sigma(s)})_{s=1,\ldots,r}$ of indexed 
operators are almost identical, only the operators $\widetilde{f}_{\sigma(t)}$ 
and $\widetilde{f}_{\sigma(t+1)}$ are swapped. This proves the formula $\sigma' = \sigma s_t$, and we
have $\ell(\sigma') = \ell(\sigma) \pm 1$. Let $\widetilde{f}_{\sigma(t)} = f_{a,h}$ and $\widetilde{f}_{\sigma(t+1)} = f_{b,k}$. Note that $\ui'_t >_{\lex} \ui_t$ because $\ui_t$ is a right peak in $\mC$. Hence $k\leq h$ and, by (iii) of  Lemma \ref{lemma_indexedOperators}, we get $\sigma(t+1) > \sigma(t)$. So, by Lemma \ref{lemma_weakRightOrder}, we have $\sigma'>_R\sigma$ and, in particular, $\sigma'>\sigma$. We conclude also that $\ell(\sigma') = \ell(\sigma) + 1$ as claimed.
\end{proof}

\begin{proposition}\label{proposition_operatorPermutation2}
If $\mC = F_{t_1}F_{t_2}\cdots F_{t_{k-1}}F_{t_k}(\mCright)$, then $\sigma_\mC = s_{t_k}s_{t_{k-1}}\cdots s_{t_2}s_{t_1}$ and this is a reduced expression in $\mathsf S_{r}$.
\end{proposition}
\begin{proof} 
Since $\sigma_{\mCright}$ is the identity permutation,
$\sigma_\mC = s_{t_k}s_{t_{k-1}}\cdots s_{t_2}s_{t_1}$
by Proposition~\ref{proposition_operatorPermutation2}. And since each operator $F_j$, $j = 1,\ldots,k$, increases the length by $1$, the expression $s_{t_k}s_{t_{k-1}}\cdots s_{t_2}s_{t_1}$ has length $k$ and is hence reduced.
\end{proof}
\subsection{The image of the chains in the symmetric group}
In this section  we describe the image of $C(d,n)$ in $\mathsf{S}_r$. We 
set for short $I_0 := I(d, n-1)$ and $I_1 := I(d-1,n-1) * n$, so that we 
have the disjoint union $I(d,n) = I_0 \sqcup I_1$.

By Lemma~\ref{lemma_maximalChainContainingSubMax}, the interval $[\ui_{\max,n-1},\ui_{\max}]$ is independent of the choice of the order: Bruhat or lexicographic order, and it is totally ordered with respect to both orders. We denote by
\[
\uj := (n-d)(n-d+1) \cdots (n-2)n = f_{n-1,d}(\ui_{\max,n-1})\, > \ui_{\max,n-1}
\]
the second smallest element.
\begin{lemma}\label{lemma_j}
\begin{itemize}
\item[(i)] The row $\uj$ is the minimal element in $\mCleft\cap I_1$.
\item[(ii)] If $\ui\in I(d,n)$ is such that $\ell(\ui) \geq \ell(\uj)$, then $\ui\in I_1$.
\end{itemize}
\end{lemma}
\begin{proof}
The property (i) is clear. In order to prove (ii) let $\ui = i_1i_2\cdots i_d$. If $\ui\in I_0$, i.e. if $i_d < n$, 
then $i_h\leq n+h-d-1$ for all $h = 1,\ldots,d$. So
\[
\ell(\ui) = \sum_{h = 1}^d (i_h - h) \leq \sum_{h = 1}^d (n - d - 1) = d(n-d-1) < d(n-d-1) + 1 = \ell(\uj),
\]
and we get a contradiction.
\end{proof}

\begin{proposition}\label{propositin_isARightPeak}
Let $\mC:\ui_r > \ui_{r-1} > \cdots > \ui_1 > \ui_0$ be a maximal chain in $I(d,n)$.
\begin{itemize}
\item[(i)] If $\mC\neq\mCleft$ and $t$ is maximal such that $\ui_t\not\in\mCleft$, then $\ui_t$ is a right peak in $\mC$. In particular $F_t(\mC)\neq 0$.
\item[(ii)] If $\mC\neq\mCright$ and $t$ is minimal such that $\ui_t\not\in\mCright$, then $\ui_t$ is a left peak in $\mC$. In particular $E_t(\mC)\neq 0$.
\end{itemize}
\end{proposition}
\begin{proof} We proceed by induction on $n$ to (i). The proof of (ii) is analogous.

Let $k$ be such that $\ui_h \in I(d,n-1)$ for all $h = 0, 1, \ldots,k$ and $\ui_h \in I(d-1,n-1)*n$ for all 
$h = k+1, k+2, \ldots,r$; see Proposition \ref{proposition_chainDecomposition}. Recall also that 
$\ui_{k+1} = f_{n-1,d}(\ui_k)$.

We consider first the case $t\leq k$. We have $\ui_{k+1}\in\mCleft\cap I_1$. Since $\uj$ is the 
minimal element of $\mCleft\cap I_1$ (Lemma \ref{lemma_j}), we deduce that $\uj\leq\ui_{k+1}$. 
If $\uj<\ui_{k+1}$, then $\ell(\uj)\leq\ell(\ui_k)$ and $\ui_k\in I_1$ by Lemma \ref{lemma_j}, 
which is impossibile by the definition of $k$. So $\ui_{k+1} = \uj$ and
$\ui_k = e_{n-1,d}(\ui_{k+1}) = \ui_{\max,n-1}\in \mCleft$. It follows: $t< k$.

In particular $\mC':\ui_k > \ui_{k-1} > \cdots > \ui_0$ is a maximal chain in $I_0\subset I(d,n)$. Then, by 
induction on $n$, we get that $\ui_t$ is a right peak in $\mC'$, but then $\ui_t$ is a right peak also in 
$\mC$ since being a right peak is a local property depending only on the interval $[\ui_{t-1}, \ui_{t+1}]$.

Now suppose that $t \geq  k + 1$. Consider the chain $\ui_{\max} = \ui_r > \ui_{r-1} > \cdots > \ui_{k+1}$; 
it is contained in $I_1 = I(d-1,n-1)*n$, a poset isomorphic to $I(d-1,n-1)$. Hence we can complete it to a 
maximal chain $\mC'$ in $I_1$ by adding elements on the right:
\[
\mC':\ui_r > \cdots > \ui_{k+1} > \ui'_k > \ui'_{k-1} > \cdots > \ui'_h = 123\cdots (n-2)n = \min I_1.
\]
Since $t\geq k+1$, the element $\ui_t$ is still maximal in $\mC'$ with the property that it is not an element 
in the leftmost chain of $I_1$. So, by induction on $n$, $\ui_t$ is a right peak in $\mC'$. If $t \geq k + 2$ 
then $\ui_t$ is a right peak also in $\mC$ since $\ui_{t-1}$ and $\ui_{t+1}$ are elements of both $\mC'$ 
and $\mC$ and being a right peak is a local property of the interval $[\ui_{t-1}, \ui_{t+1}]$.

Finally suppose $t = k+1$. We claim that $\ui_{k+2} = \uj$. One has $\ui_{k+2}\in\mCleft\cap I_1$ 
by the maximality of $t$. Since $\uj$ is minimal in $\mCleft\cap I_1$ by Lemma \ref{lemma_j}, it follows $\uj\leq \ui_{k+2}$.

We have $\ui_k\leq \ui_{\max,n-1}$ because $\ui_{\max,n-1}$ is the maximal element of $I_0$. 
Since $\uj$ is the second smallest element in the interval $[\ui_{\max,n-1},\ui_{\max}]$ one knows
$\ui_{\max,n-1} < \uj$. So we have the inequalities
$\ui_k \leq \ui_{\max,n-1} < \uj \leq \ui_{k+2}$.

If $\ui_k = \ui_{\max,n-1}$, then $\uj\in\mC$ because $[\ui_{\max,n-1},\ui_{\max}]$ is totally ordered 
and $\uj$ is an element of this interval. By Lemma~\ref{lemma_maximalChainContainingSubMax},
the interval $[\ui_{\max,n-1},\ui_{\max}]$ is contained in  $\mCleft$. But $\ui_k = \ui_{\max,n-1} < \uj \leq \ui_{k+2}$
and $\ell(\ui_k)=k$, $\ell(\ui_{k+2})=k+2$ implies either $\ui_t=\ui_{k+1}=\uj$ or  
$\ui_k = \ui_{\max,n-1} <\ui_{k+1}< \uj$, which is impossible by the choice of $t$.
So we have refined our inequalities to $\ui_k < \ui_{\max,n-1} < \uj \leq \ui_{k+2}$,
and $i_{k+1}\not=\ui_{\max,n-1}$ because $ \ui_{\max,n-1} \in \mCleft$.

But $\ell(\ui_k) = k$ and $\ell(\ui_{k+2}) = k+2$, hence we cannot have also $\uj < \ui_{k+2}$. 
This proves our claim that $\ui_{k+2} = \uj$.
So the interval $[\ui_{t-1},\ui_{t+1}] = [\ui_k,\ui_{k+2}]$ is given by the elements
\[
\ui_k < \ui_{\max,n-1}, \ui_{k+1} < \uj = \ui_{k+2}.
\]
Note that $\ell(\ui_{\max,n-1}) = k+1 = \ell(\ui_{k+1})$, hence $\ui_{\max,n-1}>_{\lex}\ui_{k+1}$ by Lemma \ref{lemma_maximalChainContainingSubMax} and we have proved that $\ui_t = \ui_{k+1}$ is a right peak in $\mC$.
\end{proof}

\begin{coro}\label{corollary_operatorAllZero}
Let $\mC$ be a maximal chain in $I(d,n)$.
\begin{itemize}
\item[(i)] If $F_t(\mC) = 0$ for each $0\leq t\leq r+1$, then $\mC = \mCleft$.
\item[(ii)] There exist $t_1, t_2, \ldots, t_k$ such that $F_{t_1}F_{t_2}\cdots F_{t_k}(\mC) = \mCleft$,
\item[(iii)] If $E_t(\mC) = 0$ for each $0\leq t\leq r+1$, then $\mC = \mCright$.
\item[(iv)] There exist $t_1, t_2, \ldots, t_k$ such that $E_{t_1}E_{t_2}\cdots E_{t_k}(\mC) = \mCright$.
\end{itemize}
\end{coro}
\begin{proof} (i) and (iii) follows at once by Proposition~\ref{propositin_isARightPeak}; we show that (ii), (iv) is proved in a similar way.

If $\mC = \mCleft$ then the claim is trivially true. If $\mC\neq \mCleft$, then there exists $t$ such that $F_t(\mC)\neq 0$ by (i). But $F_t(\mC) >_{\lex} \mC$, hence in a finite number of steps we must reach the greatest chain, i.e. $\mCleft$.
\end{proof}
\begin{proposition}\label{proposition_risingNonZero}
Let $\sigma = \sigma_\mC$ and let $1\leq t\leq r-1$ be such that $\sigma s_t <_R \sigma$. Then $E_t(\mC)\neq 0$.
\end{proposition}
\begin{proof} 
Let $\mC:\ui_r > \cdots > \ui_0$ and consider the elements $\ui_{t+1} > \ui_t > \ui_{t-1}$. Let $\widetilde{f}_{\sigma(t)} = f_{a,h}$ and $\widetilde{f}_{\sigma(t+1)} = f_{b,k}$; along $\mC$ we have
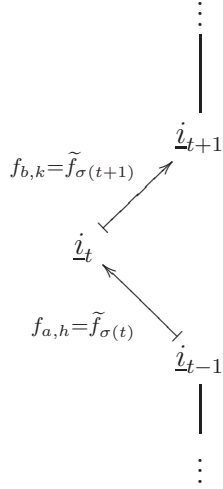
\[
\xymatrix{
 & \vdots\ar@{-}[d]\\
 & \ui_{t+1}\\
 \ui_t\ar@{|->}[ur]^{f_{b,k} = \widetilde{f}_{\sigma(t+1)}}\\
 & \ui_{t-1}\ar@{|->}[ul]^{f_{a,h} = \widetilde{f}_{\sigma(t)}}\\
 & \vdots\ar@{-}[u].
}
\]
In other words, $\ui_{t}$ is obtained by increasing the $h$--th entry of $\ui_{t-1}$, which is an $a$, and $\ui_{t+1}$ is obtained by increasing the $k$--th entry of $\ui_t$, which is a $b$.

By Lemma \ref{lemma_weakRightOrder} $\sigma(t) > \sigma(t+1)$ since $\sigma s_t <_R\sigma$ and then, by  Lemma \ref{lemma_indexedOperators} (iii), either $h < k$ or $h = k$ and $a > b$.

Suppose we are in the first case. The operator $\widetilde{f}_{\sigma(t)} = f_{a,h}$ changes in $\ui_{t-1}$ only
the $h$-th entry. Since $h<k$,  the $k$--th and the $(k+1)$--th entries of $\ui_t$ are equal to the $k$--th and 
the $(k+1)$--th entries of $\ui_{t-1}$ respectively. 
In particular, the $k$--th entry is equal to $b$ and the $(k+1)$--entry are greater than $b+1$ since $f_{b,k}(\ui_t)\neq 0$. 

It follows that $\ui'_t = f_{b,k}(\ui_{t-1})\neq 0$ and $\ui_{t+1} = f_{a,h}f_{b,k}(\ui_{t-1})$, and hence $\ui_t$ is a peak in $\mathfrak C$.
Now $\ui_{t}$ is obtained from $\ui_{t-1}$ by increasing the $h$-th entry, $\ui'_{t}$ is obtained from $\ui_{t-1}$ by increasing the 
$k$-th entry. Since $h<k$, this implies $\ui_t >_{\lex} \ui'_t$, which shows that $\ui_t$ is a left peak in $\mC$.
It follows  $E_t(\mC)\neq 0$ by definition. 

It remains the case  $h = k$ and $a > b$. This case is not possible because the $h$--th entry of $f_{a,h}(\ui_{t-1})$ is $a+1 > a > b$, which implies
$\ui_{t+1} = f_{b,h}f_{a,h}(\ui_{t-1}) = 0$, leading to a contradiction.
\end{proof}
For the rest of this section we endow $\mathsf{S}_{r}$ with the weak Bruhat order. 
Let $\sigma_{\max} = \sigma_{\mCleft}$, the notation $ [e,\sigma_{\max}]_R$ is used for the interval in $ \mathsf{S}_{r}$
between the identity element and $\sigma_{\max}$ with respect to the weak Bruhat order.
\begin{theorem}\label{main_theorem}
The map $C(d,n) \longrightarrow   [e,\sigma_{\max}]_R$, defined by
$\mC \longmapsto \sigma_\mC$, 
is a bijection between the set $C(d,n)$ of all maximal chains in $I(d,n)$ and the interval $[e,\sigma_{\max}]_R$.
With respect this bijection, the operators $E_t$, $F_t$ correspond to the multiplication on the right by $s_t$, i.e.: 
if $E_t(\mC)\neq 0$ or $F_t(\mC)\neq 0$, then $\sigma_{E_t(\mC)} = \sigma_\mC s_t$ respectively
$\sigma_{F_t(\mC)} = \sigma_\mC s_t$.
\end{theorem}
\begin{proof} We show first that the image of the map $\mC\longmapsto\sigma_\mC$ is contained in $[e, \sigma_{\max}]_R$.

Let $\mC$ be a maximal chain in $I(d,n)$. By  Corollary \ref{corollary_operatorAllZero} (iv), there exist $t_1,\ldots,t_k$ such that $E_{t_1}E_{t_2}\cdots E_{t_k}(\mC) = \mCright$. Hence $\mC = F_{t_k}\cdots F_{t_2}F_{t_1}(\mCright)$, and 
$\sigma_\mC = s_{t_1}\cdots s_{t_k}$ by Proposition~\ref{proposition_operatorPermutation2}, and this expression is reduced.

By  Corollary \ref{corollary_operatorAllZero} (ii), there exist $t'_1, \ldots, t'_h$ such that $F_{t'_1}\cdots F_{t'_h}(\mC) = \mCleft$,
and hence one has $F_{t'_1}\cdots F_{t'_h}F_{t_k}\cdots F_{t_2}F_{t_1}(\mCright)= \mCleft$. It follows 
that $\sigma_{\max} = s_{t_1}\cdots s_{t_k}s_{t'_h}\cdots s_{t'_1}$, and this last expression is reduced by 
Proposition~\ref{proposition_operatorPermutation2}, and hence $\sigma_\mC\le_R \sigma_{\max}$.

Now we prove that each $\sigma \leq_R \sigma_{\max}$ is the permutation of a maximal chain in $I(d,n)$. We show this by proving that the set of all permutations associated to a maximal chain is an ideal for the partial order $\leq_R$ on $\mathsf{S}_{r}$.

Let $\sigma = \sigma_{\mC}$ and let $\sigma s_t <_R \sigma$. By 
Proposition \ref{proposition_risingNonZero}, $\mC' := E_t(\mC)\neq 0$. 
So $\mC=F_t(\mC')$ and $\sigma_{\mC'} = \sigma s_t$; so $\sigma s_t$ 
is the permutation of a chain. This proves that  the set of all permutations 
associated to a maximal chain is an ideal for the partial order $\leq_R$ on $\mathsf{S}_{r}$.
Since $\sigma_{\max}$ is in the ideal, the full interval $[e,\sigma_{\max}]_R$
is in the ideal, and hence the map $C(d,n) \longrightarrow   [e,\sigma_{\max}]_R$ is surjective.
The map is injective by Lemma~\ref{lemma_indexedOperators}, which finishes the proof of the
theorem.
\end{proof}

\section{Newton-Okounkov bodies}
We recall some facts from the theory of Newton-Okounkov bodies. Let $Y\subseteq \mathbb{P}(V)$ be an embedded projective variety.
We consider a maximal chain of subvarieties of $Y$:
$Y_\bullet:Y=Y_r\supseteq \ldots\supseteq 
Y_1\supseteq Y_0=\{p\}$, and a corresponding 
sequence of parameters $\mathbf u=(u_r,\ldots,u_1)$ 
(see Section~\ref{section:geom:valuation}). Such a pair determines a valuation $\nu_{Y_\bullet}^s$ on the field 
$\mathbb K(Y)$ with values in $\mathbb Z^r$, where the latter
is endowed with the lexicographic order.

If $Y$ is the Grassmann variety, examples of such chains 
are the maximal chains of  Schubert subvarieties considered 
in Section~\ref{section:Schubert:sequences}. The maximal chain
of subvarieties $Y_\bullet$ does not determine the 
system of parameters $\mathbf u$ uniquely. In the 
next section we will encounter an example of 
different valuations
$\nu_{\mathfrak{C}}^b$ and $\nu_{\mathfrak{C}}^s$
associated to the same maximal chain but different
systems of parameters.

We discuss in this section how a change of the system of parameters affects the valuation and the associated valuation
monoid.
In the case of the Grassmann variety, our ultimate goal is to understand more precisely how the valuation changes
if one changes the maximal sequence of Schubert subvarieties (see Section~\ref{section:Grassmann:variety}).
It turns out that changing the system of parameters is a quite helpful tool to achieve this goal.

Throughout the following, let $\mathbb{K}$ be an algebraically closed field of characteristic 0, 
and let $V$ be a finite dimensional  vector space over $\mathbb{K}$.

\subsection{Geometric valuation}\label{section:geom:valuation}

Let $Y\subseteq \mathbb{P}(V)$ be a projective variety of dimension $r$. For a projective subvariety 
$X\subseteq Y$, let $\widehat{X}$ denote its affine cone in $V$.

We fix a flag of projective subvarieties on $Y$: 
$$Y_\bullet:Y=Y_r\supseteq Y_{r-1}\supseteq\ldots\supseteq Y_1\supseteq Y_0=\{p\}$$ 
with $\dim Y_k=k$ and $p\in Y$ a point. We assume furthermore that for $k=1,\ldots,r$, $Y_k$ is smooth in codimension one. Under this assumption, each local ring $\mathcal{O}_{Y_{k-1},Y_k}$ for $k=1,\ldots,r$ is a discrete valuation ring (DVR). The following procedure of introducing a geometric valuation on $Y$ can be found in \cite[Example 2.13]{KK12}, see also \cite[Example 1.6]{Kav15}. 

First we fix a \emph{sequence of parameters} $\mathbf{u}:=(u_r,\ldots,u_1)$ on $Y$ associated to $Y_\bullet$: they are by definition rational functions on $Y$ such that for each $k$, the restriction $u_k|_{Y_k}$ is a uniformizer in the DVR $\mathcal{O}_{Y_{k-1},Y_k}$. We fix on $\mathbb{Z}^r$ the lexicographic ordering.  We write the elements of $\mathbb{Z}^r$ as row vectors. The valuation 
$$\nu_{Y_\bullet}^{\mathbf{u}}:\mathbb{K}(Y)\setminus\{0\}\to\mathbb{Z}^r$$
associated to $Y_\bullet$ and $\mathbf{u}$ is defined by: for $f\in\mathbb{K}(Y)\setminus\{0\}$, if $\nu_{Y_\bullet}^{\mathbf{u}}(f)=(a_r,\ldots,a_1)$, then:
\begin{enumerate}
\item[-] $a_r$ is the vanishing order of $f_r^{\mathbf{u}}:=f$ on $Y_{r-1}$;
\item[-] $a_{r-1}$ is the vanishing order of $f_{r-1}^{\mathbf{u}}:=(u_r^{-a_r}f_r^{\mathbf{u}})|_{Y_{r-1}}$ on $Y_{r-2}$;
\item[-] $a_{r-2}$ is the vanishing order of $f_{r-2}^{\mathbf{u}}:=(u_{r-1}^{-a_{r-1}}f_{r-1}^{\mathbf{u}})|_{Y_{r-2}}$ on $Y_{r-3}$;
\item[-] $\cdots$
\end{enumerate}
It is straightforward to verify that $\nu_{Y_\bullet}^{\mathbf{u}}$ defines indeed a valuation.

Let $\mathbb{K}[Y]$ be the homogeneous coordinate ring of $Y\hookrightarrow \mathbb{P}(V)$. It is naturally graded:
$$
\mathbb{K}[Y]=\bigoplus_{m\ge 0} \mathbb{K}[Y]_m,
$$
where $\mathbb{K}[Y]_m$ is the homogeneous component consisting of elements of degree $m$. Fix a nonzero
element $\tau\in  \mathbb{K}[Y]_1$, not vanishing in 
$Y_0=\{p\}$.
For an arbitrary valuation $\nu:\mathbb{K}(Y)\setminus\{0\}\to\mathbb{Z}^r$, the associated  \textit{valuation monoid}  
(see \cite{KK12, LM09}) is by definition
$$
\Gamma_{\nu}:=\{0\}\cup\bigcup_{m\geq 1} \left\{\left.\left(m,\nu\left(\frac{f}{\tau^m}\right)\right)\right\vert f\in\mathbb{K}[Y]_m\setminus\{0\}\right\}\subseteq\mathbb{N}\times\mathbb{Z}^r.
$$
We write $\Gamma_{\nu}(m)$ for the set of elements of degree $m$, i.e. the elements are of the form $(m,\nu({f}/{\tau^m}))$.
The associated convex \textit{Newton-Okounkov body} $\Delta_{\nu}$ is defined as
$$
\Delta_{\nu}:=\overline{\mathrm{convex}\ \mathrm{hull}\left(\left\{\frac{\mathbf{a}}{m}\mid (m,\mathbf{a})\in\Gamma_{\nu},\ m\geq 1\right\}\right)}\subseteq\mathbb{R}^r,
$$
where the closure is taken with respect to the Euclidean topology on $\mathbb{R}^r$.  

\begin{remark}
\begin{enumerate}
\item In the definition of the Newton-Okounkov body, the convex hull operator is in fact redundant.
\item The valuation monoid and the Newton-Okounkov body  depend on the choice of $\tau$ and one should write  
$\Gamma_{\nu,\tau}$ and $\Delta_{\nu,\tau}$. 
If $\sigma\in  \mathbb{K}[Y]_1$ is another nonzero element not vanishing in $Y_0=\{p\}$, then one has for a non-zero homogeneous function $f\in\mathbb{K}[Y]_m$: 
$\nu(f/\sigma^m)= \nu(f/\tau^m)+m\nu(\tau/\sigma)$, and hence:
$\Delta_{\nu,\sigma}=\Delta_{\nu,\tau}+\nu(\tau/\sigma)$,
and $\Gamma_{\nu,\sigma}(m)=\{ (m,\mathbf{a}+m\nu(\tau/\sigma))\mid  (m,\mathbf{a})\in \Gamma_{\nu,\tau}(m)\}$.

In the following, we always assume that $\tau$ has been fixed once and for all, and by abuse of notation, 
we omit the dependence and write just $\Gamma_{\nu}$ and 
$\Delta_{\nu}$.
\end{enumerate}
\end{remark}

For $Y\subseteq \mathbb{P}(V)$ consider the valuation $\nu_{Y_\bullet}^{\mathbf{u}}$ associated to the choice 
of a sequence $Y_\bullet$ of subvarieties  and a system of parameters $\mathbf{u}$. We write 
$\Gamma_{Y_\bullet}^{\mathbf{u}}$ for the associated valuation monoid and 
$\Delta_{Y_\bullet}^{\mathbf{u}}$ for the associated Newton-Okounkov body to emphasize 
the dependence on the choice of $Y_\bullet$ and $\mathbf{u}$.

\subsection{Change of sequence of parameters}

We keep the notations as in the previous subsection and let $\mathbf{u}=(u_r,\ldots,u_1)$ and $\mathbf{v}=(v_r,\ldots,v_1)$ be two sequences of parameters. We study in this subsection the relation between the two valuations $\nu_{Y_\bullet}^{\mathbf{u}}$ and $\nu_{Y_\bullet}^{\mathbf{v}}$. For this we need a few preparations.

\begin{enumerate}
\item For $k=1,\ldots,r$, since both $u_k|_{Y_k}$ and $v_k|_{Y_k}$ are uniformizers in the DVR $\mathcal{O}_{Y_{k-1},Y_k}$, there exist invertible elements $\lambda_k\in\mathcal{O}_{Y_{k-1},Y_k}^\times$ such that 
\begin{equation}\label{Eq:0}
u_k|_{Y_k}=\lambda_k\cdot v_k|_{Y_k}.
\end{equation}
\item For $k=0,\ldots,r-1$, we consider the flag of subvarieties 
$$Y_\bullet^k: Y_k\supseteq Y_{k-1}\supseteq\ldots\supseteq Y_1\supseteq Y_0=\{p\}$$
starting from $Y_k$. Then $(u_k,\ldots,u_1)$ and $(v_k,\ldots,v_1)$ are sequences of parameters associated to $Y_\bullet^k$. We obtain therefore valuations 
$$\nu_{Y_\bullet^k}^{\mathbf{u}},\ \nu_{Y_\bullet^k}^{\mathbf{v}}:\mathbb{C}(Y_k)\setminus\{0\}\to\mathbb{Z}^k\hookrightarrow\mathbb{Z}^r$$
where the last embedding is given by 
$$(a_k,\ldots,a_1)\mapsto (0,\ldots,0,a_k,\ldots,a_1)\in\mathbb{Z}^r.$$
Such an embedding is compatible with the lexicographic ordering.
\item We define a matrix $R_{\mathbf{u}}^{\mathbf{v}}\in\mathrm{Mat}_r(\mathbb{Z})$ as follows. First let $M_{\mathbf{u}}^{\mathbf{v}}\in\mathrm{Mat}_r(\mathbb{Z})$ be the matrix whose $k$-th row is $\nu_{Y_\bullet^k}^{\mathbf{v}}(\lambda_k)$. Notice that since $\lambda_k$ is invertible in $\mathcal{O}_{Y_{k-1},Y_k}$, the first $r-k+1$ entries of $\nu_{Y_\bullet^k}^{\mathbf{v}}(\lambda_k)$ are zero. Therefore if we define 
$$R_{\mathbf{u}}^{\mathbf{v}}:=\mathrm{I}_n+M_{\mathbf{u}}^{\mathbf{v}}\in\mathrm{Mat}_r(\mathbb{Z}),$$
it is clear that $R_{\mathbf{u}}^{\mathbf{v}}$ is an upper triangular unimodular matrix, \emph{i.e.}, a matrix with integer coefficients with determinant $1$.
\end{enumerate}

\begin{proposition}\label{Prop:BaseChange}
For $f\in\mathbb{K}(Y)\setminus\{0\}$, we have
$$
\nu_{Y_\bullet}^{\mathbf{u}}(f)\cdot R_{\mathbf{u}}^{\mathbf{v}}=\nu_{Y_\bullet}^{\mathbf{v}}(f).
$$
\end{proposition}

\begin{proof}
We start by rephrasing the definition: denote for $s=1,\ldots,r$
$$
\nu_{Y_\bullet}^{\mathbf{v}}(\lambda_s)=(0,\ldots,0,\mu_{s,s-1},\mu_{s,s-2},\ldots,\mu_{s,1}).
$$
For $\mathbf{v}=(v_r,\ldots,v_1)$ and $p=1,\ldots,s-2$, one has:
$$
v_{s-p}^{-\mu_{s,s-p}}\cdots v_{s-1}^{-\mu_{s,s-1}}\lambda_s|_{Y_{s-p-1}}
$$
vanishes with order $\mu_{s,s-p-1}$ on $Y_{s-p-2}$.

Set $\nu_{Y_\bullet}^{\mathbf{u}}(f)=(a_r,\ldots,a_1)$ and $\nu_{Y_\bullet}^{\mathbf{v}}(f)=(b_r,\ldots,b_1)$.
We prove simultaneously the following two statements: for $k=1,\ldots,r$, 
\begin{equation}\label{Eq:1}
a_k+\sum_{s=k+1}^r a_s\mu_{s,k}=b_k;
\end{equation}
and for $k=1,\ldots,r-1$,
\begin{equation}\label{Eq:2}
f_k^{\mathbf{u}}=\lambda_{k+1}^{-a_{k+1}}\lambda_{k+2}^{-a_{k+2}}\cdots \lambda_r^{-a_r}v_{k+1}^{b_{k+1}-a_{k+1}}\cdots v_{r-1}^{b_{r-1}-a_{r-1}}v_r^{b_r-a_r} v_{k+1}^{-b_{k+1}}f_{k+1}^{\mathbf{v}}|_{Y_k},
\end{equation}
where $f_k^{\mathbf{u}}$ and $f_{k+1}^{\mathbf{v}}$ are the functions appearing in the definition of the valuations $\nu_{Y_\bullet}^{\mathbf{u}}(f)$ and $\nu_{Y_\bullet}^{\mathbf{v}}(f)$.
Note that the Equation \eqref{Eq:1} is the statement in the proposition.

The proofs of the statements are by downward induction on $k$. If $k=r$, both $a_k$ and $b_k$ are the vanishing orders of $f$ on $Y_{r-1}$, hence $a_r=b_r$, which is the starting point for the proof by induction for \eqref{Eq:1}. 
For subsequent equation, i.e. $k=r-1$, it follows by the definition of the $\lambda_j$, $j=1,\ldots,r$ in  \eqref{Eq:0}:
$$
f_{r-1}^{\mathbf{u}}=u_r^{-a_r}f|_{Y_{r-1}}=\lambda_r^{-a_r}v_r^{-a_r}f|_{Y_{r-1}}=\lambda_r^{-a_r}v_r^{-b_r}f_r^{\mathbf{v}}|_{Y_{r-1}},
$$
which is the starting point for the proof by induction for \eqref{Eq:2}.  We first prove the Equation \eqref{Eq:2}. By definition,
\begin{eqnarray*}
f_k^{\mathbf{u}}&=&u_{k+1}^{-a_{k+1}}f_{k+1}^{\mathbf{u}}|_{Y_k}\\
&=& \lambda_{k+1}^{-a_{k+1}}v_{k+1}^{b_{k+1}-a_{k+1}}v_{k+1}^{-b_{k+1}}f_{k+1}^{\mathbf{u}}|_{Y_k},
\end{eqnarray*}
then apply the induction hypothesis to $f_{k+1}^{\mathbf{u}}$.  We have
$$
f_{k+1}^{\mathbf{u}}=\lambda_{k+2}^{-a_{k+2}}\lambda_{k+3}^{-a_{k+3}}\cdots \lambda_r^{-a_r}v_{k+2}^{b_{k+2}-a_{k+2}}\cdots v_{r-1}^{b_{r-1}-a_{r-1}}v_r^{b_r-a_r} v_{k+2}^{-b_{k+2}}f_{k+2}^{\mathbf{v}}|_{Y_{k+1}}.
$$
But $f_{k+1}^{\mathbf{v}}=v_{k+2}^{-b_{k+2}}f_{k+2}^{\mathbf{v}}|_{Y_{k+1}}$ by definition, which finishes the proof of \eqref{Eq:2}.
To show Equation \eqref{Eq:1}, we start by rewriting Equation \eqref{Eq:2} in the following form using the induction hypothesis for Equation \eqref{Eq:1}:
$$f_k^{\mathbf{u}}=g_1\cdots g_{r-k} v_{k+1}^{-b_{k+1}}f_{k+1}^{\mathbf{v}}|_{Y_k},$$
where 
$$g_p:=(v_{k+1}^{-\mu_{k+p,k+1}}\cdots v_{k+p-1}^{-\mu_{k+p,k+p-1}}\lambda_{k+p})^{-a_{k+p}}.$$
It remains to study the vanishing behavior of each part in the above formula:
\begin{enumerate}
\item[-] by definition, $f_k^{\mathbf{u}}$ vanishes of order $a_k$ in $Y_{k-1}$;
\item[-] it follows from the discussion at the beginning of the proof that $g_p$ vanishes of order $-a_{k+p}\mu_{k+p,k}$ in $Y_{k-1}$;
\item[-] $v_{k+1}^{-b_{k+1}}f_{k+1}^{\mathbf{v}}|_{Y_k}$ is by definition $f_k^{\mathbf{v}}$, hence it vanishes of order $b_k$ in $Y_{k-1}$.
\end{enumerate}
Summing up the vanishing multiplicities terminates the proof of Equation \eqref{Eq:1}.
\end{proof}

The following corollary follows immediately from the definition of the valuation monoid and the 
Newton-Okounkov body. Let 
$\tilde R_{\mathbf{u}}^{\mathbf{v}}\in 
\mathrm{Mat}_{r+1}(\mathbb{Z})$ be the 
matrix having as entries 
in the first row and the first column 
only zeros except for a $1$ on the diagonal, 
and the remaining
$r\times r$ submatrix is equal to 
$R_{\mathbf{u}}^{\mathbf{v}}$.
\begin{coro}\label{Cor:BaseChange}
The valuation monoids and the Newton-Okounkov bodies satisfy
$$
\Gamma_{Y_\bullet}^{\mathbf{v}}=\{\mathbf{a}
\cdot \tilde R_{\mathbf{u}}^{\mathbf{v}}\mid \mathbf a\in \Gamma_{Y_\bullet}^{\mathbf{u}}\}.
\quad\textrm{and}\quad
\Delta_{Y_\bullet}^{\mathbf{u}}\cdot R_{\mathbf{u}}^{\mathbf{v}}=\Delta_{Y_\bullet}^{\mathbf{v}}.
$$
\end{coro}

As a summary, changing sequences of parameters results in an upper triangular unimodular transformation between the valuations and the Newton-Okounkov bodies.

\section{Schubert valuations on Grassmann varieties}\label{section:Grassmann:variety}

The standard reference for basic facts on Grassmann varieties is \cite{LB}.
Let $\mathrm{Gr}_{d,n}$ be the Grassmann variety of $d$-dimensional subspaces in $\mathbb{K}^n$, viewed as a projective variety via the Pl\"ucker embedding $\mathrm{Gr}_{d,n}\subseteq \mathbb{P}(\Lambda^d\mathbb{K}^n)$. Fixing the standard coordinate basis $e_1,\ldots,e_n$ of $\mathbb{K}^n$, the set $\{e_{\underline{i}}\mid \underline{i}\in I(d,n)\}$ forms a basis of $\Lambda^d\mathbb{K}^n$ where for $\underline{i}=i_1\cdots i_d$, $e_{\underline{i}}:=e_{i_1}\wedge\ldots\wedge e_{i_d}\in\Lambda^d\mathbb{K}^n$. Let $P_{\underline{i}}:=e_{\underline{i}}^*$ be the dual basis. They are called Pl\"ucker coordinates on $\mathrm{Gr}_{d,n}$. 

The Grassmann variety is also a homogeneous space: 
Let $\mathrm{SL}_n(\mathbb K)$ be the group of $n\times n$ matrices with 
entries in $\mathbb{K}$ of determinant $1$. Let $B_n$ be the subgroup of upper triangular matrices in $\mathrm{SL}_n(\mathbb K)$, 
and let $P_d$ be the parabolic subgroup of block diagonal matrices in $\mathrm{SL}_n$ with two blocks of sizes $d$ respectively $n-d$. 
The Grassmann variety $\mathrm{Gr}_{d,n}$ is isomorphic to the homogeneous space $\mathrm{SL}_n(\mathbb K)/P_d$. 

For $1\leq k\leq n$, denote by $f_k$ the element of the Lie algebra $\frak{sl}_n(\mathbb K)$: $f_k:=E_{k+1,k}$. Here, by convention, 
$E_{i,j}\in M_n(\mathbb K)$ is the matrix whose only non-zero entry is a $1$ as the $(i,j)$-entry. 

In this section we study two types of valuations on $\mathbb K(\mathrm{Gr}_{d,n})$, both depend on the choice
of a maximal chain of Schubert subvarieties. One type comes from the setup of birational sequences \cite{FFL17} and the other 
from the framework of Seshadri stratification \cite{CFL1, CFL2}. Our goal is to deduce a combinatorial formula for the 
latter with the help of birational sequences, and to understand the affect on the valuations in case one changes the maximal chains.

\subsection{Valuations from birational sequences}

\subsubsection{Birational sequences and valuations}

Birational sequences are introduced in \cite{FFL17} in order to give a unified approach to some 
known toric degenerations of flag varieties. 

We fix a maximal chain in the poset $I(d,n)$
$$\mathfrak{C}:\ \underline{i}_r>\ldots>\underline{i}_0=\underline{i}_{\mathrm{min}}$$ 
with $r=d(n-d)$ and consider the following flag of subvarieties in $\mathrm{Gr}_{d,n}$:
$$\mathrm{Gr}_{d,n}=X(\underline{i}_r)\supseteq X(\underline{i}_{r-1})\supseteq\ldots\supseteq X(\underline{i}_0)$$
where $X(\underline{i}_k)$ is the Schubert variety 
$$
X(\underline{i}_k):=\overline{B_n\cdot e_{\underline{i}_k}}\subseteq\mathrm{Gr}_{d,n}
\hookrightarrow \mathbb{P}(\Lambda^d\mathbb{K}^n).
$$
Consider the action of the Lie algebra $\frak{sl}_n(\mathbb K)$ on $\Lambda^d\mathbb K^n$.
We denote by $\alpha_{\ell_1},\ldots,\alpha_{\ell_r}$ the simple roots with $1\leq \ell_1,\ldots,\ell_r\leq n-1$ which are 
uniquely determined by  the following property: for $k=1,\ldots,r$, 
$$
f_{\ell_k}\cdot e_{\underline{i}_{k-1}}=e_{\underline{i}_{k}}.
$$
We call the sequence of simple roots $(\alpha_{\ell_r},\ldots,\alpha_{\ell_1})$ a \emph{birational sequence} for $\mathrm{Gr}_{d,n}$. The choice of the name comes from the following construction.
We consider the map
$$m:\mathbb{A}^r\to U_{-\ell_r}\times\ldots\times U_{-\ell_1}\to \mathrm{Gr}_{d,n}\hookrightarrow \mathbb{P}(\Lambda^d\mathbb{K}^n),$$
$$(t_r,\ldots,t_1)\mapsto (\exp(t_rf_{\ell_r}),\ldots,\exp(t_1f_{\ell_1}))\mapsto \exp(t_rf_{\ell_r})\cdots \exp(t_1f_{\ell_1}) e_{\underline{i}_0}.$$

We recall the following well-known properties of the map $m$ (see, for example, \cite{FFL17}):

\begin{proposition}\label{Prop:Bir}
\begin{enumerate}
\item The map $m:\mathbb{A}^r\to\mathrm{Gr}_{d,n}$ is birational, hence 
$$\mathbb{K}(\mathrm{Gr}_{d,n})\cong\mathbb{K}(t_r,\ldots,t_1).$$
\item When $t_r=\ldots=t_{k+1}=0$, the image of $m$ in $\mathrm{Gr}_{d,n}$ is contained in the Schubert variety $X(\underline{i}_k)$. Moreover, $m$ induces a field isomorphism $\mathbb K[X(\underline{i}_k)]\cong \mathbb{K}(t_k,\ldots,t_1)$.
\end{enumerate}
\end{proposition}

We fix on $\mathbb{N}^r$ the following monomial ordering $>_{\mathrm{grlex}}$: for $\mathbf{a}=(a_r,\ldots,a_1)$ and $\mathbf{b}=(b_r,\ldots,b_1)\in\mathbb{N}^r$, we declare that $\mathbf{a}>_{\mathrm{grlex}}\mathbf{b}$ if
\begin{enumerate}
\item[-] $a_r+\ldots+a_1>b_r+\ldots+b_1$,
\item[-] or $a_r+\ldots+a_1=b_r+\ldots+b_1$ and $a_r>b_r$,
\item[-] or $a_r+\ldots+a_1=b_r+\ldots+b_1$, $a_r=b_r$ and $a_{r-1}>b_{r-1}$,
\item[-] and so on $\cdots$.
\end{enumerate}

This birational map allows us to define a valuation on $\mathbb{K}(\mathrm{Gr}_{d,n})$. First we define a valuation 
$$\nu:\mathbb{K}[t_r,\ldots,t_1]\setminus\{0\}\to(\mathbb{N}^r,>_{\mathrm{grlex}}),$$
$$f=\sum_{\mathbf{a}\in\mathbb{N}^r}\lambda_{\mathbf{a}}t^{\mathbf{a}}\mapsto \mathrm{min}_{>_{\mathrm{grlex}}}\{\mathbf{a}\in\mathbb{N}^r\mid \lambda_{\mathbf{a}}\neq 0\}.$$
The valuation $\nu$ extends to a valuation $\nu:\mathbb{K}(t_r,\ldots,t_1)\setminus\{0\}\to\mathbb{Z}^r$ by: for $f,g\in\mathbb{K}[t_r,\ldots,t_1]$,
$$\nu\left(\frac{f}{g}\right):=\nu(f)-\nu(g).$$

According to Proposition \ref{Prop:Bir}(1), we obtain a valuation
$$\nu^b_{\mathfrak{C}}:\mathbb{K}(\mathrm{Gr}_{d,n})\setminus\{0\}\to\mathbb{Z}^r.$$

The homogeneous coordinate ring $\mathbb{K}[\mathrm{Gr}_{d,n}]$ is embedded into $\mathbb{K}(\mathrm{Gr}_{d,n})$ by sending a homogeneous element $f\in\mathbb{K}[\mathrm{Gr}_{d,n}]$ of degree $m$ to 
$$\frac{f}{P_{\underline{i}_{\mathrm{min}}}^m}\in \mathbb{K}(\mathrm{Gr}_{d,n}).$$

As an advantage, the valuation monoid $\Gamma_{\nu_{\mathfrak{C}}}^b$ and hence the Newton-Okounkov body $\Delta_{\nu_{\mathfrak{C}}}^b$ associated to a valuation coming from birational sequences can be computed with the help of the representation theory.

First notice that by definition, for $\underline{j}\in I(d,n)$, the image of the rational function $P_{\underline{j}}/P_{\underline{i}_{\mathrm{min}}}$ in $\mathbb{K}(t_r,\ldots,t_1)$ is given by:
\begin{equation}\label{Eq:BirPj}
P_{\underline{j}}(\exp(t_rf_{\ell_r})\cdots\exp(t_1f_{\ell_1})\cdot e_1\wedge\ldots\wedge e_d).
\end{equation}

\begin{remark}
Note that for any $\underline{j}\in I(d,n)$, this rational function is indeed a polynomial in $\mathbb{K}[t_r,\ldots,t_1]$, and each monomial appearing in this polynomial is square-free.
\end{remark}

\begin{example}\label{Ex:ValPik}
The valuation 
$$\nu^b_{\mathfrak{C}}(P_{\underline{i}_k}/P_{\underline{i}_{\mathrm{min}}})=(0,\ldots,0,1,\ldots,1)$$
where there are $k$ copies of $1$. Indeed, from the choice of the sequence of simple roots $(\alpha_{\ell_r},\ldots,\alpha_{\ell_1})$ it follows that the monomial $t_k\cdots t_1$ appears in the image of $P_{\underline{i}_k}/P_{\underline{i}_{\mathrm{min}}}$ in $\mathbb{K}(t_r,\ldots,t_1)$ with a non-zero coefficient. Now it suffices to notice that this monomial is indeed minimal with respect to $>_{\mathrm{grlex}}$ because from the above remark, the image of $P_{\underline{i}_k}/P_{\underline{i}_{\mathrm{min}}}$ in $\mathbb{K}(t_r,\ldots,t_1)$ is contained in $\mathbb{K}[t_r,\ldots,t_1]$, so there are no negative coordinates.
\end{example}

\subsubsection{Explicit formulae for $\mathfrak{C}_{\mathrm{right}}$}

We start by determining the Newton-Okounkov body for a special flag of subvarieties associated to the maximal chain $\mathfrak{C}_{\mathrm{right}}$ in $I(d,n)$. The notation $\mathfrak{C}_{\mathrm{right}}$ will also be used for the flag of subvarieties in $\mathrm{Gr}_{d,n}$ corresponding to the maximal chain. The birational sequence associated to this flag of subvarieties is given by
$$(\alpha_{n-d},\ldots,\alpha_{1},\alpha_{n-d+1},\ldots,\alpha_{2},\ldots,\alpha_{n-1},\ldots,\alpha_{d}).$$
Such a sequence contains $d$ blocks of simple roots, in the $k$-th block the index of the simple roots starts from $n-d+k-1$ and descends until $k$.

\begin{proposition}\label{Prop:ValCright}
Let $\underline{j}=j_1\ldots j_d$ and 
$$\nu_{\mathfrak{C}_{\mathrm{right}}}^b(P_{\underline{j}})=(a_{1,n-d},\ldots,a_{1,1},a_{2,n-d},\ldots,a_{2,1},\ldots,a_{d,n-d},\ldots,a_{d,1}).$$
Then for $k=1,\ldots,d$, $a_{k,1}=\ldots=a_{k,{j_k-k}}=1$, and all other coordinates are $0$.
\end{proposition}

\begin{proof}
It follows from Equation \eqref{Eq:BirPj} that $a_{i,j}\in \{0,1\}$ for $1\leq i\leq d$ and $1\leq j\leq n-d$.

First note that with the choice of $a_{i,j}$ in the statement, we have
\begin{equation}\label{Eq:Action}
f_{n-d}^{a_{1,n-d}}\cdots f_{1}^{a_{1,1}}\cdots f_{n-1}^{a_{d,n-d}}\cdots f_d^{a_{d,1}}\cdot e_1\wedge\ldots\wedge e_d=e_{j_1}\wedge\ldots \wedge e_{j_d}.
\end{equation}

Now we examine block by block. Notice that for any simple root $\alpha_k$ not appearing in the first block, $f_k\cdot e_1=0$. If $j_1=1$ we move to the next block. Otherwise we have to use the root vectors $f_{j_p-1}\cdots f_1$ to move $e_1$ to $e_{j_p}$. Then $\nu_{\mathfrak{C}_{\mathrm{right}}}^b(P_{\underline{j}})$ starts with the exponents $(0,\ldots,0,1,\ldots,1)$ where there are $j_p-1$ copies of $1$ and $n-d-j_p+1$ copies of $0$. Since the valuation takes the minimal exponent with respect to $>_{\mathrm{grlex}}$, the Equation \eqref{Eq:Action} implies that $p=1$. 

We move to the second block. Again if $j_2=2$ then we move to the next block. If not, the only $\alpha_k$ such that $f_k\cdot e_2\neq 0$ are the two $\alpha_2$ appearing in the first two blocks. From the minimality assumption with respect to $>_{\mathrm{grlex}}$, if we could make a choice for $\alpha_k$ in the first two blocks, we would prefer the one in the second block since it produces smaller exponents with respect to $>_{\mathrm{grlex}}$. Therefore the same argument as in the first block implies that the second block of $\nu_{\mathfrak{C}_{\mathrm{right}}}^b(P_{\underline{j}})$ reads $(0,\ldots,0,1,\ldots,1)$ where there are $j_2-1$ copies of $1$.

Repeating this argument for all blocks terminates the proof.
\end{proof}

As a consequence, the valuations of Pl\"ucker coordinates are all different.

\begin{example}
For $\mathrm{Gr}_{3,6}$ the birational sequence associated to $\mathfrak{C}_{\mathrm{right}}$ is given by 
$$(\alpha_3,\alpha_2,\alpha_1,\alpha_4,\alpha_3,\alpha_2,\alpha_5,\alpha_4,\alpha_3).$$ 
The valuation 
$$\nu_{\mathfrak{C}_{\mathrm{right}}}^b(P_{245})=(0,0,1,0,1,1,0,1,1).$$
\end{example}

\subsubsection{Newton-Okounkov bodies}

We start by identifying the Newton-Okounkov body of the valuation associated to $\mathfrak{C}_{\mathrm{right}}$. To simplify notations we write temporarily $\Gamma$ and $\Delta$ for the valuation monoid and the Newton-Okounkov body associated to $\nu_{\mathfrak{C}_{\mathrm{right}}}^b$.

In order to describe $\Delta$, we recall the definition of an order polytope associated to a finite poset $(P,\succ)$. The order polytope $\mathcal{O}(P,\succ)$ associated to $(P,\succ)$ is by definition the set of order preserving maps $P\to [0,1]$, where $[0,1]$ is endowed with the usual order. In notation, 
$$\mathcal{O}(P,\succ):=\{\mathbf{x}=(x_p)_{p\in P}\in [0,1]^P|\ x_p\leq x_q \text{ for }p\prec q\}.$$

The following properties of $\mathcal{O}(P,\succ)$ are proved in \cite{Sta86}:
\begin{enumerate}
\item The set of integer points in $\mathcal{O}(P,\succ)$ coincides with its vertices; moreover, they are characteristic functions $\chi_F$ of filters $F$ in $P$. Here a filter $F$ is understood as a subset of $P$ such that if $p\in F$ and $q\succ p$ then $q\in F$.
\item The volume of $\mathcal{O}(P)$ equals to the number of linear extensions of $P$, divided by $n!$, where $n$ is the cardinality of $P$.
\end{enumerate}

On the set
$$P_{d,n}:=\{(i,j)\mid 1\leq i\leq d,\ \ 1\leq j\leq n-d\}$$
we define a partial order by requiring $(i,j)\succeq (k,\ell)$ if and only if $i\geq k$ and $j\leq\ell$.

For $1\leq k\leq d$, let 
$$P_{d,n}^k:=P_{d,n}\cap \{(k,1),(k,2),\ldots,(k,n-d)\}.$$

\begin{proposition}\label{Prop:OrderPolytope}
The Newton-Okounkov body $\Delta$ coincides with the order polytope $\mathcal{O}(P_{d,n},\succ)$.
\end{proposition}

\begin{proof}
We first show that $\mathcal{O}(P_{d,n},\succ)\subseteq [0,1]^{P_{d,n}}$ can be embedded into $\Delta$ under the map
$$\iota:[0,1]^{P_{d,n}}\to \mathbb{R}^r,$$
$$(x_{i,j})_{(i,j)\in P_{d,n}}\mapsto (x_{1,n-d},\ldots,x_{1,1},x_{2,n-d},\ldots,x_{2,1},\ldots,x_{d,n-d},\ldots,x_{d,1}).$$

Take a vertex $\chi_F$ for a filter $F\subseteq P_{d,n}$, we denote 
$$F\cap P_{d,n}^k=\{(k,1),\ldots,(k,s_k)\}$$
for $0\leq s_k\leq n-d$. It follows from Proposition \ref{Prop:ValCright} that 
$$\nu_{\mathfrak{C}_{\mathrm{right}}}^b(P_{s_1+1,s_2+2,\ldots,s_d+d})=\iota(\chi_F).$$

It remains to show that the volume of $\Delta$ equals to that of $\mathcal{O}(P_{d,n},\succ)$. By \cite[Theorem 13.6]{CFL1}, the degree of $\mathrm{Gr}_{d,n}$ in the Pl\"ucker embedding equals to the number of maximal chains in the poset $I(d,n)$, divided by $r!$. By \cite{KK12,LM09}, up to the $1/r!$ factor, the volume of the Newton-Okounkov body computes the degree of the projective variety in the fixed embedding, and hence this number is equal to the volume of $\Delta$. On the other hand, the cardinality of $P_{d,n}$ is $d(n-d)=r$, hence it remains to show that the number of maximal chains in $I(d,n)$ coincides with the linear extensions of $P_{d,n}$. This is proved in \cite[Lemma 2.2]{FL17} by investigating the distributive lattice structure on $I(d,n)$ and by looking $P_{d,n}$ as the set of join-irreducible elements therein.
\end{proof}

Newton-Okounkov bodies associated to other maximal chains in $I(d,n)$ can be obtained from $\Delta$ by permuting coordinates.

Let $\mathfrak{C},\mathfrak{D}$ be two maximal chains in $I(d,n)$ such that $F_j(\mathfrak{C})=\mathfrak{D}$ for some $1\leq j\leq r-1$. Then the associated birational sequences are related in the following way: if the one associated to $\mathfrak{C}$ is $(\alpha_{i_1},\ldots,\alpha_{i_r})$, then that associated to $\mathfrak{D}$ is 
$$(\alpha_{i_1},\ldots,\alpha_{i_{j-1}},\alpha_{i_{j+1}},\alpha_{i_j},\alpha_{i_{j+2}},\ldots,\alpha_{i_r}).$$ 
Here notice that the Lie bracket of the root vectors satisfy: $[f_{i_j},f_{i_{j+1}}]=0$.

\begin{proposition}\label{Prop:2-move}
Let $\tau_j:\mathbb{R}^r\to\mathbb{R}^r$ be the linear map swapping the coordinates $j$ and $j+1$. Then 
$$\tau_j(\Gamma_{\nu_{\mathfrak{C}}}^b)=\Gamma_{\nu_{\mathfrak{D}}}^b\ \ and\ \ \tau_j(\Delta_{\nu_{\mathfrak{C}}}^b)=\Delta_{\nu_{\mathfrak{D}}}^b.$$
\end{proposition}

\begin{proof}
It follows from $[f_{i_j},f_{i_{j+1}}]=0$ that the root subgroups $\displaystyle U_{-\alpha_{i_j}}$ and  $U_{-\alpha_{i_{j+1}}}$ commute, the proposition then follows.
\end{proof}

In Definition \ref{definition_chainPermutation} we have introduced for each maximal chain $\mathfrak{C}$ in $I(d,n)$ a permutation $\sigma(\mathfrak{C})\in \mathsf{S}_{r}$, we will denote it by $\sigma_{\mathfrak{C}}\in \mathsf{S}_r$. It acts on $\mathbb{R}^r$ by permuting coordinates. 

Combining Proposition \ref{proposition_operatorPermutation2}, Proposition \ref{Prop:OrderPolytope} and Proposition \ref{Prop:2-move} gives:

\begin{coro}\label{Coro:Prop:2-move}
\begin{enumerate}
\item Let $\mathfrak{C}$ be a maximal chain in $I(d,n)$. Then $\Delta_{\nu_{\mathfrak{C}}}^b=\sigma_\mathfrak{C}^{-1}(\Delta)$.
\item For any maximal chain $\mathfrak{C}$ in $I(d,n)$, the Newton-Okounkov body $\Delta_{\nu_{\mathfrak{C}}}^b$ is unimodular equivalent to the order polytope $\mathcal{O}(P_{d,n},\succ)$.
\end{enumerate}
\end{coro}

\subsection{Valuations from Seshadri stratification}

In the work \cite{CFL1} of the authors, we introduced a different valuation whose definition incorporates data in a Seshadri stratification. In this subsection we recall this construction for the Seshadri stratification defined in \cite[Example 2.6]{CFL1} on the Grassmann variety $\mathrm{Gr}_{d,n}$.

Fix a flag of subvarieties in $\mathrm{Gr}_{d,n}$:
$$\mathfrak{C}:\ \mathrm{Gr}_{d,n}=X(\underline{i}_r)\supseteq X(\underline{i}_{r-1})\supseteq\ldots\supseteq X(\underline{i}_0)$$
and take $f\in\mathbb{K}[\mathrm{Gr}_{d,n}]\setminus\{0\}$, we define a sequence of rational functions
$$(f_r=f,f_{r-1},\ldots,f_1,f_0)$$ 
and a sequence of integers $(a_r,\ldots,a_1,a_0)$ as follows:
\begin{enumerate}
\item[-] $a_r$ is the vanishing order of $f_r$ on $\widehat{X}(\underline{i}_{r-1})$ and $f_{r-1}:=P_{\underline{i}_r}^{-a_r} f_r|_{\widehat{X}(\underline{i}_{r-1})}$; 
\item[-] $a_{r-1}$ is the vanishing order of $f_{r-1}$ on $\widehat{X}(\underline{i}_{r-2})$ and $f_{r-2}:=P_{\underline{i}_{r-1}}^{-a_{r-1}} f_{r-1}|_{\widehat{X}(\underline{i}_{r-2})}$; 
\item[-] $\cdots$;
\item[-] $a_1$ is the vanishing order of $f_{1}$ on the affine line $\widehat{X}(\underline{i}_{0})$, $f_{0}:=P_{\underline{i}_{1}}^{-a_{1}} f_{1}|_{\widehat{X}(\underline{i}_{0})}$;
\item[-] $a_0$ is the vanishing order of $f_0$ at $0\in \Lambda^d\mathbb{K}^n$.
\end{enumerate}

\begin{remark}
The situation in \cite{CFL1} got largely simplified by choosing the Pl\"ucker coordinates as extremal functions in Example 2.6 of \emph{loc.cit.}: since the vanishing order of $P_{\underline{i}}$ on $X(\underline{i})$ is one, all bonds appearing in the above example are one, and hence the important number $N$ in \emph{loc.cit.} is one.
\end{remark}

It is proved in \cite[Proposition 6.10]{CFL1} that 
$$\mathbb{K}[\mathrm{Gr}_{d,n}]\setminus\{0\}\to\mathbb{Z}^{r+1},\ \ f\mapsto (a_r,\ldots,a_1,a_0)$$
is a valuation. Note that when $f$ is a homogeneous element of degree $m$, it follows from \cite[Corollary 7.5]{CFL1} that $a_r+\ldots+a_1+a_0=m$. Therefore for homogeneous elements, the last coordinate is uniquely determined by the first $r$ coordinates. From now on we will forget $a_0$ and denote
$$\nu_{\mathfrak{C}}^s:\mathbb{K}[\mathrm{Gr}_{d,n}]\setminus\{0\}\to\mathbb{Z}^{r},\ \ \ \ f\mapsto (a_r,\ldots,a_1).$$

The goal of this section is to give a combinatorial formula for $\nu_{\mathfrak{C}}^s(P_{\underline{j}})$ for $\underline{j}\in I(d,n)$ and describe explicitly the associated Newton-Okoukov body.

\subsection{Combinatorial formula}

We keep notations in the previous subsections: after fixing a flag of subvarieties $X(\underline{i}_r)\supseteq X(\underline{i}_{r-1})\supseteq \ldots\supseteq X(\underline{i}_1)\supseteq X(\underline{i}_0)$ on $\mathrm{Gr}_{d,n}$ arising from a maximal chain $\mathfrak{C}:\underline{i}_r>\ldots>\underline{i}_1>\underline{i}_0$ in $I(d,n)$, we obtained an isomorphism of fields $\mathbb{C}(\mathrm{Gr}_{d,n})\cong\mathbb{C}(t_r,\ldots,t_1)$.

To simplify notations, we will write $X_k:=X(\underline{i}_k)$.

First notice that both $t_k|_{X_k}$ and $(P_{\underline{i}_k}/P_{\underline{i}_0})|_{X_k}$ are uniformizers in $\mathcal{O}_{X_{k-1},X_k}$. 
For the first one this follows by Proposition~\ref{Prop:Bir}, and for the second one this follows by the fact that the vanishing order of $P_{\underline{i}_k}|_{X_k}$ on $X_{k-1}$ is one, and $P_{\underline{i}_0}$ is not identically zero on $X_{k-1}$.
Therefore 
$$\mathbf{u}:=(P_{\underline{i}_r}/P_{\underline{i}_0},\ldots,P_{\underline{i}_1}/P_{\underline{i}_0})\ \ \text{and}\ \ \mathbf{v}:=(t_r,\ldots,t_1)$$
are sequences of parameters associated to the flag of subvarieties.

For $k=1,\ldots,r$, let $\lambda_k\in\mathcal{O}_{X_{k-1},X_k}^\times$ be such that $(P_{\underline{i}_k}/P_{\underline{i}_0})|_{X_{k}}=\lambda_k\cdot t_k|_{X_k}$. Here we can choose 
$$\lambda_k:=\left.\frac{P_{\underline{i}_k}}{P_{\underline{i}_0}t_k}\right|_{X_k}.$$
It is a rational function on $\mathrm{Gr}_{d,n}$ which does not vanish identically on $X_{k-1}$.

According to Proposition \ref{Prop:BaseChange}, for any $f\in\mathbb{K}[\mathrm{Gr}_{d,n}]\setminus\{0\}$, 
$$\nu_{\mathfrak{C}}^s(f)=\nu_{\mathfrak{C}}^b(f)\cdot (R_{\mathbf{u}}^{\mathbf{v}})^{-1}.$$
By definition, the $k$-th row of the matrix $R_{\mathbf{u}}^{\mathbf{v}}$ is $\ve_k+\nu_{\mathfrak{C}}^b(\lambda_k)$ where $\{\ve_1,\ldots,\ve_r\}$ stands for the standard coordinate basis of $\mathbb{Z}^r$. Since $\nu_{\mathfrak{C}}^b(t_k)=\ve_k$, it follows:
\begin{eqnarray*}
\ve_k+\nu_{\mathfrak{C}}^b(\lambda_k)&=& \ve_k+\nu_{\mathfrak{C}}^b(P_{\underline{i}_k})-\nu_{\mathfrak{C}}^b(P_{\underline{i}_0}t_k)\\
&=& (0,\ldots,0,1,\ldots,1)
\end{eqnarray*}
where there are $k$ copies of $1$. Here in the last equality we have applied Example \ref{Ex:ValPik}. The matrices $R_{\mathbf{u}}^{\mathbf{v}}$ and its inverse take the following forms:
$$R_{\mathbf{u}}^{\mathbf{v}}=\begin{pmatrix}
1 & 1 & \cdots & \cdots & 1 & 1\\
0 & 1 & \cdots & \cdots & 1 & 1\\
\vdots & \ddots & \ddots & & \vdots & \vdots\\
\vdots & & \ddots & \ddots & \vdots & \vdots\\
\vdots & & & \ddots & 1 & 1\\
0 & \cdots & \cdots & \cdots & 0 & 1
\end{pmatrix}\ \ \text{and}\ \ (R_{\mathbf{u}}^{\mathbf{v}})^{-1}=
\begin{pmatrix}
1 & -1 & 0 & \cdots & \cdots & 0\\
0 & 1 & -1 & \ddots & & \vdots\\
\vdots & \ddots & \ddots & \ddots & \ddots & \vdots\\
\vdots & & \ddots & \ddots & -1 & 0\\
\vdots & & & \ddots & 1 & -1\\
0 & \cdots & \cdots & \cdots & 0 & 1
\end{pmatrix}.$$
Therefore if $\nu_{\mathfrak{C}}^b(f)=(a_r,\ldots,a_1)$, then 
$$\nu_{\mathfrak{C}}^s(f)=(a_r,a_{r-1}-a_r,a_{r-2}-a_{r-1},\ldots,a_1-a_2).$$

Similarly, by Corollary \ref{Cor:BaseChange}, the valuation lattices and the Newton-Okounkov bodies associated to these two valuations satisfy:
\begin{coro}\label{corollary:b:and:s}
We have:
$$
\Gamma_{\nu_{\mathfrak{C}}}^s=\Gamma_{\nu_{\mathfrak{C}}}^b\cdot (R_{\mathbf{u}}^{\mathbf{v}})^{-1},\quad 
\Delta_{\nu_{\mathfrak{C}}}^s=\Delta_{\nu_{\mathfrak{C}}}^b\cdot (R_{\mathbf{u}}^{\mathbf{v}})^{-1}.
$$
\end{coro}
\begin{example}\label{Ex:explicit}
For $\mathfrak{C}=\mathfrak{C}_{\mathrm{right}}$ we can deduce a closed formula for the valuations $\nu_{\mathfrak{C}}^s(P_{\underline{j}})$ for the Pl\"ucker coordinates.

From Proposition \ref{Prop:ValCright}, for $\underline{j}=j_1\ldots j_d$, we obtain the following closed formula:
\begin{eqnarray*}
\nu_{\mathfrak{C}_{\mathrm{right}}}^b(P_{\underline{j}})=(a_{1,n-d},a_{1,n-d-1}-a_{1,n-d},&\ldots&,a_{1,1}-a_{1,2},a_{2,n-d}-a_{1,1},\ldots,a_{2,1}-a_{2,2},\\
&\ldots&,a_{d,n-d}-a_{d-1,1},\ldots,a_{d,1}-a_{d,2}),
\end{eqnarray*}
where for $k=1,\ldots,d$, $a_{k,1}=\ldots=a_{k,{j_k-k}}=1$, and all other coordinates are $0$.
\end{example}

According to this example, for an arbitrary $\mathfrak{C}$, $\Delta_{\mathfrak{C}}^s$ can be obtained from $\Delta$ by permuting coordinates. 

\begin{example}\label{last:example}
In $\mathrm{Gr}_{3,6}$ we fix the flag of subvarieties corresponding to the following maximal chain in $I(3,6)$:
$$456>356>346>246>245>235>135>134>124>123.$$
The associated birational sequence is 
$$(\alpha_3,\alpha_4,\alpha_2,\alpha_5,\alpha_3,\alpha_1,\alpha_4,\alpha_2,\alpha_3).$$
It is straightforward to show that  $\nu_{\mathfrak{C}}^b(P_{156})=(0,1,0,1,1,0,1,1,1)$, and hence $\nu_{\mathfrak{C}}^s(P_{156})=(0,1,-1,1,0,-1,1,0,0)$.

We provide another approach using the above example. One has immediately
$$\nu_{\mathfrak{C}_{\mathrm{right}}}^b(P_{156})=(0,0,0,1,1,1,1,1,1).$$
The permutation $\sigma$ sending $\mathfrak{C}$ to $\mathfrak{C}_{\mathrm{right}}$ is
$$\sigma=s_3s_6\sigma_{\mathrm{min}}=s_2s_5s_7s_4s_3s_5s_6.$$
Therefore 
$$\nu_{\mathfrak{C}}^b(P_{156})=\sigma(\nu_{\mathfrak{C}_{\mathrm{right}}}^b(P_{156}))=(0,1,0,1,1,0,1,1,1)$$
and hence we obtain 
$$\nu_{\mathfrak{C}}^s(P_{156})=(0,1,-1,1,0,-1,1,0,0).$$
\end{example}


\begin{thebibliography}{99}
\bibitem{BW}
A. Bj\"orner, M. Wachs, \emph{Bruhat order of Coxter groups and shellability}, 
Adv. in Math., 43(1), 87--100 (1982).

\bibitem{CFL1} 
R. Chiriv\`i, X. Fang, P. Littelmann, 
\emph{Seshadri stratifications and standard monomial theory}, 
Invent. Math., 234, 489--572 (2023).

\bibitem{CFL2}
R. Chiriv\`i, X. Fang, P. Littelmann, 
\emph{Seshadri stratification for Schubert varieties and Standard Monomial Theory}, 
Proc Math Sci 132, 74 (2022). Special Issue in Memory of Professor C S Seshadri.

\bibitem{FFL17}
 X. Fang, G. Fourier, and P. Littelmann. 
 \emph{Essential bases and toric degenerations arising from birational sequences}, 
 Adv. Math, Volume 312, 25 May 2017, Pages 107--149.

\bibitem{FaFoL}
X.~Fang, G.~Fourier, and P.~Littelmann, 
\emph{On toric degenerations of flag varieties}, 
in Representation Theory -- Current Trends and Perspectives, edited by H. Krause \textit{et al}, Series of Congress Reports, EMS, 2017.

\bibitem{FL17}
X. Fang, P. Littelmann, 
\emph{From standard monomial theory to semi-toric degenerations via Newton-Okounkov bodies},
Trans. Moscow Math. Soc. 2017, 275--297.

\bibitem{FN}
N. Fujita, S. Naito,
\emph{Newton-Okounkov convex bodies of Schubert varieties and polyhedral realizations of crystal bases},
Math. Z. 285 (2017), no. 1-2, 325–-352.

\bibitem{Hum}
J. E. Humphreys, 
\emph{Reflection Groups and Coxeter Groups},
Cambridge University Press, New York, (1990).


\bibitem{Kav15} 
K.~Kaveh, 
\emph{Crystal bases and Newton-Okounkov bodies},
Duke Math. J. 164 (2015), 2461--2506.

\bibitem{KK12} 
K.~Kaveh, A. G.~Khovanskii, \emph{Newton-Okounkov bodies, semigroups of integral points, graded algebras and intersection theory}, Ann. of Math. (2) 176 (2012), no. 2, 925--978.

\bibitem{LB}
V. Lakshmibai, J. Brown,
\emph{The Grassmannian variety: geometric and representation-theoretic aspects},
Dev. Math., Vol. 42. Springer, New York (2015).

\bibitem{LM09}
R. Lazarsfeld and M. Musta\c{t}$\breve{\rm a}$,
\emph{Convex bodies associated to linear series},
Ann. Sci. \'Ec. Norm. Sup\'er. (4) 42 (2009), no. 5, 783--835. 

\bibitem{RW}
K. Rietsch, L. Williams, 
\emph{Newton-Okounkov bodies, cluster duality, and mirror symmetry for Grassmannians},
Duke Math. J. 168 (2019), no. 18, 3437--3527.

\bibitem{Sta86}
R.~P. Stanley,
\emph{Two poset polytopes},
Discrete Comput. Geom., 1(1):9--23, 1986.

\end{thebibliography}
\end{document}